\documentclass{amsart}
\usepackage{amsmath,amsfonts,amsthm,amssymb,indentfirst,epic,url,graphics}

\newtheorem{theorem}{Theorem}
\newtheorem{lemma}[theorem]{Lemma}
\newtheorem{corollary}[theorem]{Corollary}
\newtheorem{proposition}[theorem]{Proposition}
\newtheorem{convention}[theorem]{Convention}
\newtheorem{definition}[theorem]{Definition}

\renewcommand{\r}{\mathrm}

\DeclareRobustCommand{\lang}{\begin{picture}(5,7)
\put(1.2,2.5){\rotatebox{45}{\line(1,0){6.0}}}
\put(1.2,2.5){\rotatebox{315}{\line(1,0){6.0}}}
\end{picture}\kern.16em}
\DeclareRobustCommand{\rang}{\kern.1em\begin{picture}(5,7)
\put(.1,2.5){\rotatebox{135}{\line(1,0){6.0}}}
\put(.1,2.5){\rotatebox{225}{\line(1,0){6.0}}}
\end{picture}}

\newcommand{\xppy}[1]{{\setlength{\mathsurround}{0em}%
\ref{d.1=:xpqpy|->}\kern.1em$_{#1}$}} 
\newcommand{\xptp}[1]{{\setlength{\mathsurround}{0em}%
\ref{d.1=:xpqtp|->}\kern.1em$_{#1}$}}
\newcommand{\xpTP}[1]{{\setlength{\mathsurround}{0em}%
\ref{d.1=:xpqtp|->}\kern.1em$'_{#1}$}}

\begin{document}

\title%
[Universal inner inverses]%
{Adjoining a universal inner inverse to a ring element}%
\thanks{
Archived at \url{http://arxiv.org/abs/1505.02312}\,.
After publication of this note, updates, errata, related references
etc., if found, will be recorded at
\url{http://math.berkeley.edu/~gbergman/papers}
}


\author{George M. Bergman}
\address{University of California\\
Berkeley, CA 94720-3840, USA}
\email{gbergman@math.berkeley.edu}

\begin{abstract}
Let $R$ be an associative
unital algebra over a field $k,$ let $p$ be an element
of $R,$ and let $R'=R\lang q\mid pqp=\nolinebreak p\rang.$
We obtain normal forms for elements of $R',$ and for elements
of $\!R'\!$-modules arising by extension of scalars
from $\!R\!$-modules.
The details depend on where in the chain
$pR\cap Rp \subseteq pR\cup Rp \subseteq pR + Rp \subseteq R$
the unit $1$ of $R$ first appears.

This investigation is motivated by a hoped-for application
to the study of the possible forms of the monoid of isomorphism
classes of finitely generated projective modules over a von~Neumann
regular ring; but that goal remains distant.

We end with a normal form result for the algebra obtained by
tying together a $\!k\!$-algebra $R$ given with a nonzero element $p$
satisfying $1\notin pR+Rp$ and a $\!k\!$-algebra $S$ given
with a nonzero $q$
satisfying $1\notin qS+Sq,$ via the pair of relations $p=pqp,$ $q=qpq.$
\end{abstract}
\maketitle

\section{Motivation: monoids of projective modules}\label{S.motivation}
It is known that the abelian monoid of isomorphism classes
of finitely generated projective modules over a general ring is subject
to no nonobvious restrictions -- the obvious restrictions being
\begin{equation}\begin{minipage}[c]{35pc}\label{d.conical}
no two nonzero elements of the monoid have sum zero,
\end{minipage}\end{equation}
and
\begin{equation}\begin{minipage}[c]{35pc}\label{d.o-}
the monoid has an element $u$
such that every element is a summand in $nu$ for
some positive integer $n.$
\end{minipage}\end{equation}
(Namely, $u$ is the isomorphism class of the free module of rank~$1.)$

Indeed, every abelian monoid $M$ satisfying~\eqref{d.conical}
and~\eqref{d.o-} is known to be the monoid of
finitely generated projective modules of some hereditary
$\!k\!$-algebra, for any field $k.$
(This was proved for finitely generated $M$ in
\cite[Theorem~6.2]{cPu}, while Theorem~6.4 of that paper claimed
to show that if one weakened `hereditary'
to `semihereditary', the assumption that $M$ was finitely generated
could be dropped.
The argument indeed gave a $\!k\!$-algebra $R$ having $M$ as its monoid
of finitely generated projectives, but the proof that
$R$ was semihereditary was incorrect.
However, in \cite[Theorem~3.4]{u_deriv&} it is shown
that the $R$ so constructed
is not merely semihereditary, but hereditary.
For a similar result, see \cite[Corollary~4.5]{A+G3}.)

Recall that a ring $R$ is called {\em von~Neumann regular}
if every element $p\in R$ has an {\em inner inverse},
that is, an element $q\in R$ satisfying $pqp=p.$
The monoid of isomorphism classes of finitely generated
projective modules over a von~Neumann regular
ring is known to satisfy not only~\eqref{d.conical}
and~\eqref{d.o-}, but a strong additional restriction,
the {\em Riesz refinement property} \cite{separative}:
\begin{equation}\begin{minipage}[c]{35pc}\label{d.refinement}
If $A_0\oplus A_1\cong B_0\oplus B_1,$ then there exist
$C_{ij}$ $(i,j\in\{0,1\})$\\
such that
$A_i\cong C_{i0}\oplus C_{i1}$ and $B_i\cong C_{0i}\oplus C_{1i};$
\end{minipage}\end{equation}
that is, any such isomorphism $A_0\oplus A_1\cong B_0\oplus B_1$
can be written in the trivial form
\begin{equation}\begin{minipage}[c]{35pc}\label{d.trivO+}
$(C_{00}\oplus C_{01})\oplus (C_{10}\oplus C_{11})\ \cong
\ (C_{00}\oplus C_{10})\oplus (C_{01}\oplus C_{11}).$
\end{minipage}\end{equation}

Until a couple of decades ago, it was an open question
whether~\eqref{d.conical}-\eqref{d.refinement} completely
characterized the monoids of finitely
generated projectives of von~Neumann regular rings.
Then F.\,Wehrung \cite{FW_card}
constructed a monoid of cardinality $\aleph_2$
satisfying~\eqref{d.conical}-\eqref{d.refinement}
which cannot occur as such a monoid of projectives.
More recently, he has given an example of a {\em countable} monoid
satisfying~\eqref{d.conical}-\eqref{d.refinement}
which does not occur in this way
for any von~Neumann regular algebra over an
{\em uncountable} field \cite[\S4]{Ara}.
It remains open whether every {\em countable}
monoid satisfying~\eqref{d.conical}-\eqref{d.refinement}
is the monoid of finitely generated projectives
of {\em some} von~Neumann regular ring.

But there is in fact a strong condition, not implied
by~\eqref{d.conical}-\eqref{d.refinement}, which is not known to
fail in any von~Neumann regular ring:
\begin{equation}\begin{minipage}[c]{35pc}\label{d.separative}
$A\oplus A\cong A\oplus B\cong B\oplus B\ \implies\ A\cong B.$
\end{minipage}\end{equation}
An abelian monoid satisfying~\eqref{d.separative}.
is called {\em separative}.
A positive answer to the question of
whether the monoid of finitely generated projectives
of every von~Neumann regular ring is separative
would solve several other questions about
such rings \cite{separative}.
We remark that it is known \cite{FW}, \cite[\S4]{A+E}
that every monoid satisfying~\eqref{d.conical}
and \eqref{d.o-} can be embedded in one that also
satisfies~\eqref{d.refinement} (which
can be taken countable if the original monoid was).
Hence, applying this to monoids for which~\eqref{d.separative} fails,
one sees that there do exist abelian
monoids satisfying~\eqref{d.conical}-\eqref{d.refinement}
but not~\eqref{d.separative}.
For more on these questions, see \cite{Ara}, \cite{A+E},
\cite{A+G}, \cite{A+G2},~\cite{separative}.

Now it is known that many universal constructions
on $\!k\!$-algebras make
only ``obvious'' changes in the structure of the monoid of
finitely generated projectives \cite{cP}, \cite{cPu}, \cite{u_deriv&}.
This suggests that to investigate the possible structures of those
monoids for von~Neumann regular $\!k\!$-algebras,
we could start with a general $\!k\!$-algebra,
recursively adjoin universal inner inverses to its elements
till it becomes von~Neumann regular, and see what
conditions this process forces on the monoid of projectives.

That plan has not proved as easy as I hoped.
We obtain below normal forms for elements
of the $\!k\!$-algebra $R'=R\lang q\mid pqp=p\rang$ and for elements
of modules $M\otimes_R\,R';$ but it is not clear whether these can be
used to get useful results on isomorphism classes of modules.

The descriptions of the algebra $R'$ will show surprising differences,
depending on how near to invertible the element $p\in R$ to which we
adjoin a universal inner inverse is.
Below, we begin with a case that is challenging enough to
illustrate our method without being excessively difficult,
the case where $p$ is farthest from invertible,
namely, where $1\notin pR+Rp$~(\S\ref{S.norm}).
We then quickly cover the easy cases where $1\in pR$
and/or $1\in Rp,$ i.e., where $p$ is left or right invertible, or
both (\S\ref{S.1-sided}).
Finally, we treat the surprisingly difficult intermediate case
where $1\in pR+Rp,$ but $1\notin pR\cup Rp$~(\S\ref{S.1=:norm}).
We also examine the particular instance of this
construction where $R$ is the Weyl algebra~(\S\ref{S.Weyl}).
The last main results of the paper~(\S\ref{S.mutual})
concern a variant of the above constructions,
in which the pair of relations $pqp=p,$ $qpq=q,$ is used to
join together two given $\!k\!$-algebras.

For reasons to be noted in~\S\ref{S.enough?},
the difficult results of~\S\ref{S.1=:norm} (and the easy
results of~\S\ref{S.1-sided}) may be less useful
than the results of~\S\ref{S.norm}; so some readers may wish
to skip or skim them.
A list of the sections of this note containing the most important
material, in this light, along with some others, noted in
curly brackets, that are less essential but not very difficult, is:
\S\S\ref{S.defs}
\ref{S.norm} \{\ref{S.digress}\}
\ref{S.modules} \{\ref{S.enough?}
\ref{S.1-sided}\}
\ref{S.mutual} \{\ref{S.further}\}.

Incidentally, though, as noted above, there exist monoids
satisfying conditions~\eqref{d.conical}-\eqref{d.refinement}
but not condition~\eqref{d.separative},
no ``concrete'' examples of such monoids appear
to be known, but only constructions which obtain them by
starting with a monoid satisfying neither~\eqref{d.refinement}
nor~\eqref{d.separative}, universally
adjoining elements $C_{ij}$ as required by~\eqref{d.refinement},
and repeating this construction transfinitely -- i.e.,
the analog of the way non-separative von~Neumann regular rings might be
constructed if the plan suggested above is successful.
It would, of course, be of interest to have
concrete examples in both the monoid and the algebra situations.

I am indebted to P.\,Ara, T.\,Y.\,Lam, N.\,Nahlus, and, especially, to
K.\,O'Meara for helpful comments, corrections
and suggestions regarding this note.

\section{Generalities}\label{S.defs}
All rings will here be associative with unit; and
the rings for which we will study the construction of
universal inner inverses will be algebras over a fixed field $k.$
If $R$ is a nonzero $\!k\!$-algebra, we will identify
the $\!k\!$-subspace of $R$ spanned by $1$ with $k.$

I am using the term ``inner inverse'' (at the advice of
T.\,Y.\,Lam) for what I had previously known as a ``quasi-inverse'',
since the latter term also has a different, better-established sense.
(Elements $x$ and $y$ of a not necessarily unital
ring $R$ are called quasi-inverses in that sense
if $xy=yx=-x-y;$ in other words, if on adjoining a unit to $R,$ one
gets mutually inverse elements $1+x$ and $1+y.$
We will not consider that concept here.
On the other hand, the choice of the letter $q$ for universal inner
inverses below is based on my having used ``quasi-inverse'' in early
drafts of this note, while the element whose inner inverse we
are adjoining will be denoted $p$ because of the visual matching
of the shapes of these two letters.)

Note that if $R$ is a ring of endomorphisms of an abelian group $A,$
then an inner inverse of an element $p\in R$ is an endomorphism $q$ that
takes every member of the image of $p$ to some inverse image under $p$
of that element, with no restriction on what it does to elements
not in the image of $p.$
From this it is easy to show that in the algebra of endomorphisms of any
$\!k\!$-vector space, every element has an inner inverse; so such
algebras are examples of von~Neumann regular rings.

The relation ``is an inner inverse of'' is not symmetric:
if $q$ is an inner inverse of $p,$
$p$ need not be an inner inverse of $q.$
For example, any element of any ring is an inner inverse of
$0,$ but $0$ is not an inner inverse of any nonzero element.
However, if an element $p$ has an inner inverse $q,$ we find
that $q'=qpq$ is an inner inverse of $p$ such
that $p$ {\em is} an inner inverse of $q'.$
Thus, the condition that an element of a ring have an
inner inverse is equivalent to the condition that it
have a ``mutual inner inverse''.
Even when both relations $pqp=p$ and $qpq=q$ hold, however,
$p$ does not uniquely determine $q.$
For instance, in the ring $M_2(R)$ of $2\times 2$ matrices
over any ring $R,$ any two members of $\{e_{11}+r e_{12}\mid r\in R\}$
are inner inverses of one another.

Our normal form results for algebras constructed by adjoining universal
inner inverses will be proved using the ring-theoretic
version of the Diamond Lemma, as developed in \cite[\S1]{<>}.
However, where in \cite{<>} I formalized reduction rules
as ordered pairs $(W,f),$ with $W$
a word in our given generators, and $f$ a linear combination
of words, to be substituted for occurrences of $W$
as subwords of other words, I here use the more informal
notation ``$W\mapsto f$''.
(Another formulation of the Diamond Lemma
appears as \cite[Proposition~1]{LB}.
Bokut' \cite{LB+2}, \cite{LB+PK} refers to it as ``the method of
Gr\"{o}bner-Shirshov bases''.)

Given a $\!k\!$-algebra $R$ and an element $p\in R,$ our construction
of a normal form for elements of $R\lang q\mid pqp=p\rang$ will
start with a $\!k\!$-basis for $R,$ which we shall
want to choose in a way that allows us
to see which elements of $R$ are left and/or right multiples of $p.$
In describing such a basis, it will be convenient to use

\begin{definition}\label{D.basis_rel}
If $U\subseteq V$ are $\!k\!$-vector-spaces, then a
{\em $\!k\!$-basis of $V$ relative to $U$} will mean a
subset $B\subseteq V$ with the property that every element of
$V$ can be written uniquely as the sum of an
element of $U$ and a $\!k\!$-linear combination of elements of $B.$
\end{definition}

Thus, the general basis of $V$ relative to $U$ can
be obtained by choosing a basis $B'$ of $V/U,$ and selecting
one inverse image in $V$ of each element of $B';$
or, alternatively, by choosing any direct-sum complement
to $U$ in $V,$ and taking a basis of that complement.
Clearly, the union of any $\!k\!$-basis of $U$ and any
$\!k\!$-basis of $V$ relative to $U$ is a $\!k\!$-basis of $V.$

\begin{lemma}\label{L.V_1+V_2}
Suppose $V_1,$ $V_2$ are subspaces of a vector space $W,$
and let $B_0$ be a basis of $V_1\cap V_2,$
$B_1$ a basis of $V_1$ relative $V_1\cap V_2,$ and
$B_2$ a basis of $V_2$ relative $V_1\cap V_2.$

Then $B_0,$ $B_1,$ and $B_2$ are disjoint, and their union
is a basis of $V_1+V_2.$
\textup{(}Hence $B_1$ is also a basis of $V_1+V_2$ relative
to $V_2,$ and $B_2$ a basis of $V_1+V_2$ relative to $V_1.)$

Hence if, further, $B_3$ is a basis of $W$ relative
to $V_1+V_2,$ then $B_0\cup B_1\cup B_2\cup B_3$ is
a basis of $W.$
\end{lemma}

\begin{proof}
The disjointness of $B_0,$ $B_1,$ and $B_2$ is immediate.
The fact that $B_2$ is a basis of $V_2$ relative $V_1\cap V_2$
means that its image in $V_2/(V_1\cap V_2)$ is a basis thereof.
But $V_2/(V_1\cap V_2)\cong (V_1+V_2)/V_1,$ so $B_2$ is also
a basis of $V_1+V_2$ relative to $V_1,$ hence its union with
the basis $B_0\cup B_1$ of $V_1$ is a basis of $V_1+V_2,$
giving the first assertion, and, in the
process, the parenthetical note that follows it.
The final assertion is then immediate.
\end{proof}

\section{A normal form for $R\lang q\mid pqp=p\rang$ when $1\notin pR+Rp.$}\label{S.norm}

Here is the situation we will consider first:
\vspace{.4em}
\begin{equation}\begin{minipage}[c]{35pc}\label{d.1_notin_Rp+pR}
In this section, $R$ will be a $\!k\!$-algebra,
and $p$ a fixed element of $R$ such that $1\notin pR+Rp.$
(So in particular, $R$ is nonzero.)
\end{minipage}\end{equation}

Under this assumption,
I claim we can take a $\!k\!$-basis of $R$ of the form
\begin{equation}\begin{minipage}[c]{35pc}\label{d.B=}
$B\cup\{1\}\ =\ B_{++}\cup B_{+-} \cup B_{-+}\cup B_{--}\cup \{1\},$
\end{minipage}\end{equation}
where
\begin{equation}\begin{minipage}[c]{35pc}\label{d.B_}
$B_{++}$ is any $\!k\!$-basis of $pR\cap Rp$ which,
if $p\neq 0,$ contains $p,$

$B_{+-}$ is any $\!k\!$-basis of $pR$ relative to $pR\cap Rp,$

$B_{-+}$ is any $\!k\!$-basis of $Rp$ relative to $pR\cap Rp,$

$B_{--}$ is any $\!k\!$-basis of $R$ relative to $pR+Rp+k.$
\end{minipage}\end{equation}
(Mnemonic: a $+$ on the left signals left divisibility
by $p,$ a $+$ on the right, right divisibility.)

Indeed, let $B_{++},$ $B_{+-},$ $B_{-+},$ $B_{--}$ be sets
as in~\eqref{d.B_}.
By Lemma~\ref{L.V_1+V_2}, $B_{++}\cup B_{+-} \cup B_{-+}$
will be a $\!k\!$-basis of $pR+Rp.$
By assumption, $1\notin pR+Rp,$ so
$B_{++}\cup B_{+-} \cup B_{-+}\cup\{1\}$ is a
$\!k\!$-basis of $pR+Rp+k.$
Hence bringing in the $\!k\!$-basis $B_{--}$ of $R$ relative
to that subspace gives us a $\!k\!$-basis of $R.$

Below, we will typically denote an element of $B$ by a
letter such as $x.$
However, when such an element is specified as belonging to
$B_{++}\cup B_{+-}$ (respectively, to $B_{++}\cup B_{-+}),$
we shall often find it useful to write it in a form such as $px$
(respectively, $xp).$
Note that if $p$ is a zero-divisor in $R,$ the $x$ in
such an expression will not be uniquely determined.
We could assume one such representation fixed for each
member of $B_{++}\cup B_{+-},$ but we shall
not find this necessary; rather, the uses to which we shall
put such expressions will not depend on the choice of $x.$
In particular, note that given elements $xp\in Rp$ and $py\in pR,$ the
value of $xpy$ depends only on the elements $xp$ and $py,$ not
on the choices of $x$ and $y.$
For if $xp=x'p$ and $py=py',$ then $xpy=x'py=x'py'.$

In the case of elements specified as belonging
to $B_{++},$ we will often use three representations,
$x=x'p=\nolinebreak px''.$

The construction of a normal form for $R\lang q\mid pqp=p\rang$
in this section, and of similar normal forms in
subsequent sections, involves considerations
both of {\em elements} of $\!k\!$-algebras, and of
{\em expressions} for such elements.
We shall tread the thin line between ambiguity and
cumbersome notation by making

\begin{convention}\label{Cv.exprs}
Throughout this note, when we consider a $\!k\!$-algebra $S$
generated by a set $G,$ an {\em expression} for an element $s\in S$
will mean an element of the
free $\!k\!$-algebra $k\lang G\rang$ which maps to $s$
under the natural homomorphism $k\lang G\rang\to S.$
A {\em word} or {\em monomial} will mean a
member of the free monoid generated by $G$ in $k\lang G\rang.$
Thus, in descriptions of reductions $W\mapsto f,$
the word $W$ and the expression $f$ are
understood to lie in $k\lang G\rang.$

A family of words will be said to {\em give a $\!k\!$-basis
for $S$} if the $\!k\!$-subspace of $k\lang G\rang$ spanned by that
family maps bijectively to $S$ under the above natural homomorphism;
in other words, if the family is mapped one-to-one
into $S,$ and its image is a $\!k\!$-basis of $S.$

We shall use the same symbols for elements of $k\lang G\rang$
and their images in $S,$ distinguishing these by context: in
descriptions of normal forms and reductions, our symbols
will denote elements of $k\lang G\rang,$ while
in statements that a {\em relation} holds in $S,$ they
will denote elements of $S.$
\end{convention}

In the situation at hand, the outputs of our reductions
for $R\lang q\mid pqp=p\rang$ will often have to be
expressed in terms of the operations of $R.$
For this purpose, we make the notational convention that
for any $\!k\!$-algebra expression $f$ for an element
of $R,$ we shall write $f_R$ for the unique $\!k\!$-linear combination
of elements of $B\cup\{1\}$ which gives the value of $f$ in $R.$
(Thus, when we come to reductions~\eqref{d.xy|->} and~\eqref{d.xpqpy|->}
below, the inputs will be words of lengths~$2$ and $3$
respectively, while the outputs, by this notational convention, are
$\!k\!$-linear combinations of words of lengths~$\leq 1.)$

Note also that since the monomials that span the free algebra
$k\lang B\cup\{q\}\rang$
include the empty word $1,$ and none of the reductions we will give
has $1$ as its input, $1$ will belong to the $\!k\!$-basis
described in the theorem.

We can now state and prove our normal form.

\begin{theorem}\label{T.1_notin}
Let $R$ be a $\!k\!$-algebra, $p$ an element of $R$
such that $1\notin pR+Rp,$ $B\cup\{1\}$
a $\!k\!$-basis of $R$ as in~\eqref{d.B=} and~\eqref{d.B_} above, and
\begin{equation}\begin{minipage}[c]{35pc}\label{d.R'}
$R'\ =\ R\,\lang q\mid pqp=p\rang,$
\end{minipage}\end{equation}
the $\!k\!$-algebra gotten by adjoining to $R$ a universal
inner inverse $q$ to $p.$

Then $R'$ has a $\!k\!$-basis given by the set of those
words in the generating set $B\cup\{q\}$ that contain no
subwords of the form
\begin{equation}\begin{minipage}[c]{35pc}\label{d.xy}
$xy$ \quad with $x,y\in B$
\end{minipage}\end{equation}
nor
\begin{equation}\begin{minipage}[c]{35pc}\label{d.xpqpy}
$(xp)\,q\,(py)$ \quad with $xp\in B_{++}\cup B_{-+}$
and $py\in B_{++}\cup B_{+-}\,.$
\end{minipage}\end{equation}

The reduction to the above normal
form may be accomplished by the systems of reductions
\begin{equation}\begin{minipage}[c]{35pc}\label{d.xy|->}
$xy\ \mapsto\ (xy)_R$ \quad for all $x,y\in B,$
\end{minipage}\end{equation}
and
\begin{equation}\begin{minipage}[c]{35pc}\label{d.xpqpy|->}
$(xp)\,q\,(py)\ \mapsto\ (xpy)_R$ \quad for all
$xp\in B_{++}\cup B_{-+},$ $py\in B_{++}\cup B_{+-}\,.$
\end{minipage}\end{equation}
\end{theorem}

\begin{proof}
Clearly, $R'$ is generated as a $\!k\!$-algebra by
$B\cup\{q\},$ and we see that the relations
\begin{equation}\begin{minipage}[c]{35pc}\label{d.xy=}
$xy\ =\ (xy)_R$ \quad for $x,y$ as in~\eqref{d.xy|->}
\end{minipage}\end{equation}
and
\begin{equation}\begin{minipage}[c]{35pc}\label{d.xpqpy=}
$(xp)\,q\,(py)\ =\ (xpy)_R$ \quad for $xp,\ py$ as in~\eqref{d.xpqpy|->}
\end{minipage}\end{equation}
do hold in $R'.$
Moreover, these relations are sufficient to
define $R'$ in terms of our generators.
Indeed the relations~\eqref{d.xy=} constitute a presentation of $R;$
to get the additional relation $pqp=p$ of~\eqref{d.R'},
note that if $p=0$ this is vacuous, while if $p\neq 0,$ it is
the case of~\eqref{d.xpqpy=} where $xp=p=py.$

Since~\eqref{d.xy=} and~\eqref{d.xpqpy=} give
a presentation of $R',$ the statement of the
Diamond Lemma in \cite[Theorem~1.2]{<>} tells us that
the reductions~\eqref{d.xy|->} and~\eqref{d.xpqpy|->} will yield
a normal form for $R'$ if, first of
all, they satisfy an appropriate condition
guaranteeing that repeated applications of these reductions to
any expression eventually
terminate, and if, moreover, every ``ambiguity'',
in the sense of \cite{<>}, is ``resolvable''.

The first of these conditions is immediate, since each of our reductions
replaces a word by a linear combination of shorter words;
so the partial ordering on the set of all words which
makes shorter words ``$<$'' longer words, and distinct words of equal
length incomparable, is, in the language of \cite{<>}, a semigroup
partial ordering that is compatible with our reduction system,
and has descending chain condition.

To show that all ambiguities are
resolvable, we note that there are four
sorts of ambiguously reducible words (notation explained below):
\begin{equation}\begin{minipage}[c]{35pc}\label{d.xyz}
$x\cdot y\cdot z,$ \quad where $x,y,z\in B,$
\end{minipage}\end{equation}
\begin{equation}\begin{minipage}[c]{35pc}\label{d.xpqpyz}
$(xp)\,q \cdot (py)\cdot z,$ \quad where
$xp\in B_{++}\cup B_{-+},$
$py\in B_{++}\cup B_{+-},$
$z\in B,$
\end{minipage}\end{equation}
\begin{equation}\begin{minipage}[c]{35pc}\label{d.xypqpz}
$x\cdot (yp) \cdot q\,(pz),$ \quad where
$x\in B,$
$yp\in B_{++}\cup B_{-+},$
$pz\in B_{++}\cup B_{+-},$\quad and
\end{minipage}\end{equation}
\begin{equation}\begin{minipage}[c]{35pc}\label{d.xpqyqpz}
$(xp)\,q \cdot y\cdot q\,(pz),$ \quad where
$xp\in B_{++}\cup B_{-+},$
$y=py'=y''p\in B_{++},$
$pz\in B_{++}\cup B_{+-}\,.$
\end{minipage}\end{equation}

In each of these words, I have placed dots so as
to indicate the two competing reductions applicable to the word in
question, namely, the application of one of the
reductions~\eqref{d.xy|->} or~\eqref{d.xpqpy|->}
to the product of the two strings of generators surrounding
the first dot, and the application of another such reduction
to the product of the two strings surrounding the second dot.
For example, in~\eqref{d.xpqpyz} we can either reduce
$(xp)\,q\,(py)$ using~\eqref{d.xpqpy|->}, or reduce
$(py)z$ using~\eqref{d.xy|->}.

In each case, each of our two competing reductions will,
as noted, turn the indicated expression into
a $\!k\!$-linear combination of shorter words.
Most of these new words are in turn subject to a second reduction.
(The exceptions are those that arise from an occurrence of
the empty word, $1,$ in the output of the first reduction.)
I claim that for each of \eqref{d.xyz}-\eqref{d.xpqyqpz},
after these reductions are complete,
the two resulting expressions are equal;
namely, that we get $(xyz)_R,$ $(xpyz)_R,$ $(xypz)_R$ and
$(xyz)_R,$ respectively.

I will show this first, in detail,
for the simplest case,~\eqref{d.xyz},
then in outline for the most complicated case,~\eqref{d.xpqyqpz},
then note briefly what happens in the intermediate
cases~\eqref{d.xpqpyz} and~\eqref{d.xypqpz}.

In the case of~\eqref{d.xyz}, let
\begin{equation}\begin{minipage}[c]{35pc}\label{d.xy_R}
$(xy)_R\ =\ \sum_{u\in B\cup\{1\}} \alpha_u u$ $(\alpha_u\in k).$
\end{minipage}\end{equation}
Thus, the result of the ``left-hand'' reduction of
$x\cdot y\cdot z$ is $\sum_{u\in B\cup\{1\}} \alpha_u u\,z.$
Now for $u=1,$ the empty string, we have $uz=z,$
which we can write $(uz)_R,$
while for all other $u,$ the string $uz$ can be reduced to
$(uz)_R$ by an application of~\eqref{d.xy|->}.
Hence the expression $\sum \alpha_u u\,z$ can be reduced
using~\eqref{d.xy|->} to
$\sum \alpha_u (uz)_R=(\sum \alpha_u uz)_R,$
which by~\eqref{d.xy_R} equals $(xyz)_R,$ as claimed.
By symmetry, the calculation beginning with the right-hand
reduction of $x\cdot y\cdot z$ likewise yields
$(xyz)_R,$ showing that, in the language of \cite{<>},
the ambiguity corresponding to~\eqref{d.xyz} is resolvable.

Let us now look at the case of~\eqref{d.xpqyqpz},
but without explicitly writing expressions $f_R$ as
linear combinations of basis elements,
merely understanding that they represent
such linear combinations, and that the analogs of the
reductions~\eqref{d.xy|->} and~\eqref{d.xpqpy|->}
for such linear combinations can be achieved by applying~\eqref{d.xy|->}
or~\eqref{d.xpqpy|->} respectively
to each word in the linear expression.

Writing $y$ in~\eqref{d.xpqyqpz} as $py',$ we see that the result of
applying~\eqref{d.xpqpy|->} to $(xp)\,q\,(py')$ is $(xpy')_R,$
so the left-hand reduction of
$(xp)\,q\,y\,q\,(pz)$ gives $(xpy')_R\,q\,(pz).$
Using now the fact that in~\eqref{d.xpqyqpz},
$py'=y''p,$ we can rewrite this as $(xy''p)_R\,q\,(pz).$
Since $xy''p$ is right-divisible by $p,$
$(xy''p)_R$ is a $\!k\!$-linear combination of elements
of $B_{++}\cup B_{-+},$ so we can
apply~\eqref{d.xpqpy|->} to each term of this expression,
and get $(xy''pz)_R,$ in other words, $(xyz)_R.$
Again, by symmetry the calculation beginning with the right-hand
reduction gives the same result.

The cases~\eqref{d.xpqpyz} and~\eqref{d.xypqpz} combine
features of the above two.
In the former, for instance, the reader is invited to verify
that whether we begin with the
reduction of $(xp)\,q\,(py)$ or of $(py)\,z,$ a following application
of reductions of the other sort brings us to the common
answer $(xpyz)_R.$
In this case, the two parts of
the verification are not left-right dual to each other;
rather, the verification of~\eqref{d.xpqpyz} is left-right dual
to that of~\eqref{d.xypqpz}.

Since all our ambiguities are resolvable,
\cite[Theorem~1.2]{<>} tell us that the words in $B$
which do not have as subwords any words appearing as inputs
of reductions~\eqref{d.xy|->} or~\eqref{d.xpqpy|->}
form a $\!k\!$-basis of $R',$ as claimed.
\end{proof}

\section{A digression on algebras over non-fields.}\label{S.digress}

An immediate consequence of the above theorem is that $R$ can
be embedded in a $\!k\!$-algebra in which $p$ has an inner inverse.
However, this can be more easily seen from the fact that
$R$ embeds in the algebra of all endomorphisms of its underlying
$\!k\!$-vector-space, which is von~Neumann regular.

On the other hand, letting $K$ be a general commutative ring
(so as not to violate our convention that $k$ denotes a field),
a $\!K\!$-algebra $R$ with a specified element $p$ need not be
embeddable in a $\!K\!$-algebra in which $p$ has an inner inverse.
For instance, if $K$ is an integral domain,
$p$ a nonzero nonunit of $K,$ and $R=K/(p^2),$
we see that in $R,$ the image of $p$ is nonzero, but if
an inner inverse $q$ to $p$ is adjoined, then since $p\in K$
must remain central, we get $p=p\,q\,p=p^2q=0$ in $R'.$
The following result (which will not be used in the
sequel) shows, inter alia,
that for $K$ a general commutative ring, such problems
occur if and only if $K$ is not itself von~Neumann regular.

\begin{proposition}\label{P.KvN}
For $K$ a commutative ring, the following conditions are
equivalent.\\[.5em]
\textup{(a)} \ $K$ is von~Neumann regular.\\[.5em]
\textup{(b)} \ Every $\!K\!$-algebra $R$ can be embedded in
a von~Neumann regular $\!K\!$-algebra.\\[.5em]
\textup{(c)} \ For every ideal $I$ of $K$ and element $p\in K/I,$
the $\!K\!$-algebra $K/I$ can be
embedded in a $\!K\!$-algebra in which $p$ has
an inner inverse.\\[.5em]
\textup{(d)} \ For every $\!K\!$-module $M$ and nonzero $x\in M,$
one has $x\notin PM$ for some maximal ideal $P$ of $K.$
\end{proposition}

\begin{proof}
We shall show that
(a)$\implies$(d)$\implies$(b)$\implies$(c)$\implies$(a).

(a)$\implies$(d):
Given $M$ and $x$ as in (d),
let $P$ be maximal among proper ideals of $K$ containing the
annihilator of $x,$ and suppose by way of contradiction that
$x\in PM,$ so that $x=\sum_1^{n} a_i x_i$ with $a_i\in P,$ $x_i\in M.$
The ideal of $K$ generated by the $a_i$ will be generated
by an idempotent $e,$ since $K$ is von~Neumann
regular \cite[Theorem~1.1(a)$\!\implies\!$(b)]{KG},
so $e\in P,$ and since each $a_i$ lies in $eK,$ we have $ex=x.$
This says that $(1-e)x=0,$ so $1-e\in P$
(since $P$ contains the annihilator of $x),$ so $1=e+(1-e)\in P,$
contradicting the assumption that $P$ is proper.

(d)$\implies$(b):
Assuming (d), we shall show that for every nonzero $x\in R,$ there is a
homomorphism from $R$ to a von~Neumann regular $\!K\!$-algebra
which does not annihilate $x.$
Hence $R$ embeds in a direct product of such algebras,
which will itself be von~Neumann regular.

Given $x\in R-\{0\},$ if we regard $R$ as a $\!K\!$-module,~(d)
says that $x\notin PR$ for some maximal ideal $P$ of $K.$
Regarding $K/P$ as a field, this tells us that $x$ has nonzero
image in the $\!K/P\!$-algebra $R/PR.$
And as noted at the beginning of this section,
every algebra over a field $k$
embeds in a von~Neumann regular $\!k\!$-algebra.

(b)$\implies$(c):
Apply (b) with $R=K/I.$

(c)$\implies$(a):
Take any $p\in K,$ and apply (c) with $I=p^2 K,$
and with the image $\overline{p}$ of $p$ in $K/I$ in the role of $p.$
This gives us a $\!K\!$-algebra containing $K/I$ in which $\overline{p}$
has an inner inverse $\overline{q},$ and we compute $\overline{p}=
\overline{p}\,\overline{q}\,\overline{p}= \overline{p}^2\,\overline{q}=
0\,\overline{q}=0.$
Thus $p\in I=p^2 K,$ i.e., in $K$ we can write
$p=p^2 q,$ and since $K$ is commutative, this equals $pqp.$
Thus every $p\in K$ has an inner inverse, so $K$ is von~Neumann regular.
\end{proof}

\section{$R'\!$-modules}\label{S.modules}

Returning to the situation of $R$ a $\!k\!$-algebra, and
$p\in R$ with $1\notin pR+Rp,$ for which we have described the
extension $R'=R\lang q\mid p=pqp\rang,$
we now want to describe the $\!R'\!$-module $M\otimes_R R'$ for an
arbitrary right $\!R\!$-module $M,$ and examine such questions as
whether an inclusion of $\!R\!$-modules $M\subseteq N$ induces
an embedding of $M\otimes_R R'$ in $N\otimes_R R'.$

Our normal form for $R'$ will generalize easily
to a normal form for $M\otimes_R R',$
but we shall find that an inclusion
of $\!R\!$-modules does not necessarily induce an embedding
of $\!R'\!$-modules.
The reason is that the relation $p=pqp$ in $R'$ makes
$1-qp$ right annihilate $p,$ hence $1-qp$ also annihilates
all elements of the form $xp$ in any right $\!R'\!$-module.
We shall in fact see that in $M\otimes_R R',$
the set of elements of $M$ annihilated by $1-qp$ is precisely $Mp.$
Hence if $M$ is a submodule of an $\!R\!$-module $N,$ and there is
an element $y\in M$ which is not a multiple of $p$ in $M,$
but becomes one in $N,$ then
the map of $\!R'\!$-modules induced by the inclusion
$M\subseteq N$ will kill the nonzero element $y(1-qp).$

However, we shall find that we can describe the structure
of the $\!R'\!$-submodule of $N\otimes_R R'$ generated by $M$ wholly
in terms of the $\!R\!$-module structure of $M,$ and the set of
elements of $M$ which become multiples of $p$ in $N.$
Let us set up language and notation to handle this.
In the next definition, we do not assume $1\notin pR+Rp,$
since we will be calling on it again in sections where that
assumption does not apply.

\begin{definition}\label{D.tempered}
Let $k$ be a field, $R$ a $\!k\!$-algebra, and $p$ an element of $R.$

By a {\em $\!p\!$-tempered} right $\!R\!$-module, we shall mean
a pair $(M,M_+)$ where $M$ is a right $\!R\!$-module, and
$M_+$ is any $\!k\!$-vector-subspace of $M$ which contains the
subspace $Mp,$ is annihilated by the right annihilator of $p$
in $R,$ and is closed under multiplication by the subring
$\{x\in R\mid px\in Rp\}\subseteq R.$

A {\em morphism} of $\!p\!$-tempered right $\!R\!$-modules
$h:(M,M_+)\to (N,N_+)$ will mean an
$\!R\!$-module homomorphism $h:M\to N$
such that $h(M_+)\subseteq N_+.$
Such a morphism will be called
an {\em embedding} of $\!p\!$-tempered right $\!R\!$-modules if it is
one-to-one, and satisfies $M_+=h^{-1}(N_+).$

Finally, let $R'=R\lang q\mid p=pqp\rang.$
Then for any $\!p\!$-tempered $\!R\!$-module $(M,M_+),$
we shall denote by $(M,M_+)\otimes_{(R,p)} R'$ the
quotient of $M\otimes_R R'$ by the submodule generated by all elements
\begin{equation}\begin{minipage}[c]{35pc}\label{d.M:xqp-x}
$xqp-x$ \quad for $x\in M_+.$
\end{minipage}\end{equation}
\end{definition}

Observe that if $M_+=Mp,$ then $(M,M_+)\otimes_{(R,p)} R'$
is simply $M\otimes_R R'.$

For $B\cup\{1\}$ a $\!k\!$-basis of $R,$ and $f$
an expression representing an element of
$R,$ we shall continue to write $f_R$ for the $\!k\!$-linear expression
in elements of $B\cup\{1\}$ that gives the value of $f.$
Likewise, if we are given a $\!k\!$-basis $C$ of $M,$ then for any
expression $f$ representing an element of $M$ \textup{(}for example, any
$\!k\!$-linear combination of words each given by an element of $C$
followed by a \textup{(}possibly empty\textup{)} string of elements
of $B),$ we shall write $f_M$ for the $\!k\!$-linear expression
in elements of $C$ giving the value of $f$ in $M.$

Using the version of the Diamond Lemma for modules
in \cite[\S9.5]{<>}, let us now prove

\begin{proposition}\label{P.M_norm}
Let $k,$ $R,$ $p,$ $B,$
and $R'=R\lang q\mid p=pqp\rang$ be as in Theorem~\ref{T.1_notin}.
Let $(M,M_+)$ be a $\!p\!$-tempered right $\!R\!$-module,
let $C_+$ be a $\!k\!$-basis of $M_+,$ and let
$C_-$ be a $\!k\!$-basis of $M$ relative to $M_+,$
so that $C=C_+\cup C_-$ is a $\!k\!$-basis of $M.$

Then $(M,M_+)\otimes_{(R,p)} R'$ has $\!k\!$-basis
given by all words $w$ that are composed of an element of $C$
followed by a \textup{(}possibly empty\textup{)}
string of elements of $B\cup\{q\},$ such that $w$ contains no
subwords~\eqref{d.xy} or~\eqref{d.xpqpy}
as in Theorem~\ref{T.1_notin}, nor any subwords
\begin{equation}\begin{minipage}[c]{35pc}\label{d.M:xy}
$xy$ \quad with $x\in C$ and $y\in B$
\end{minipage}\end{equation}
or
\begin{equation}\begin{minipage}[c]{35pc}\label{d.M:xqpy}
$x\,q\,(py)$ \quad with $x\in C_+$ and $py\in B_{++}\cup B_{+-}\,.$
\end{minipage}\end{equation}

The reduction to the above normal form may be accomplished by
the system of reductions~\eqref{d.xy|->} and~\eqref{d.xpqpy|->}
given in Theorem~\ref{T.1_notin}, together with
\begin{equation}\begin{minipage}[c]{35pc}\label{d.M:xy|->}
$xy\ \mapsto\ (xy)_M$ \quad for $x\in C,$ $y\in B$
\end{minipage}\end{equation}
and
\begin{equation}\begin{minipage}[c]{35pc}\label{d.M:xqpy|->}
$x\,q\,(py)\ \mapsto\ (xy)_M$ \quad for
$x\in C_+,$ $py\in B_{++}\cup B_{+-}\,.$
\end{minipage}\end{equation}
\end{proposition}

\begin{proof}[Sketch of proof]
Let us first observe that in~\eqref{d.M:xqpy|->}, though the
basis-element $py$ may not uniquely determine $y,$ the element
$(xy)_M$ is nonetheless well-defined, since if $py$ can also
be written $py',$ then $y$ and $y'$ differ by a member of the
right annihilator of $p,$
so by the definition of $\!p\!$-tempered $\!R\!$-module,
their difference annihilates $x\in M_+.$

The relations corresponding
to the reductions~\eqref{d.xy|->}, \eqref{d.xpqpy|->},
\eqref{d.M:xy|->} and~\eqref{d.M:xqpy|->}
all hold in $(M,M_+)\otimes_{(R,p)} R'.$
Indeed, those corresponding to applications
of~\eqref{d.xy|->} and~\eqref{d.xpqpy|->}
hold by the structure of $R';$ those corresponding
to \eqref{d.M:xy|->} by the
$\!R\!$-module structure of $M,$ and those corresponding
to~\eqref{d.M:xqpy|->} because
in defining $(M,M_+)\otimes_{(R,p)} R',$
we have divided out by the submodule
generated by all elements~\eqref{d.M:xqp-x}.
And in fact, we see that the relations corresponding
to these reductions constitute a presentation
of the $\!R'\!$-module $(M,M_+)\otimes_{(R,p)} R'.$
As before, our reductions decrease the lengths of words,
so if all ambiguities of our reduction system are resolvable,
it will yield a normal form for
the $\!R'\!$-module $(M,M_+)\otimes_{(R,p)} R'.$

The ambiguities are of two sorts: the four given
by~\eqref{d.xyz}-\eqref{d.xpqyqpz}, which are resolvable by
Theorem~\ref{T.1_notin}, and the four analogous
ones in which the leftmost factor comes from $C$ rather than $B:$
\begin{equation}\begin{minipage}[c]{35pc}\label{d.M:xyz}
$x\cdot y\cdot z,$ where $x\in C,$ $y,z\in B,$
\end{minipage}\end{equation}
\begin{equation}\begin{minipage}[c]{35pc}\label{d.M:xqpyz}
$x\,q \cdot (py)\cdot z,$ where
$x\in C_+,$
$py\in B_{++}\cup B_{+-},$
$z\in B,$
\end{minipage}\end{equation}
\begin{equation}\begin{minipage}[c]{35pc}\label{d.M:xypqpz}
$x\cdot (yp) \cdot q\,(pz),$ where
$x\in C,$
$yp\in B_{++}\cup B_{-+},$
$pz\in B_{++}\cup B_{+-},$
\end{minipage}\end{equation}
\begin{equation}\begin{minipage}[c]{35pc}\label{d.M:xqyqpz}
$x\,q \cdot y\cdot q\,(pz),$ where
$x\in C_+,$
$y=py'=y''p\in B_{++},$
$pz\in B_{++}\cup B_{+-}\,.$
\end{minipage}\end{equation}

I claim~\eqref{d.M:xyz}-\eqref{d.M:xqyqpz} are resolvable by
computations analogous
to those we used for~\eqref{d.xyz}-\eqref{d.xpqyqpz},
the common forms to which the results of the two
possible reductions lead now being $(xyz)_M$
for~\eqref{d.M:xyz} and~\eqref{d.M:xqpyz},
$(xypz)_M$ for~\eqref{d.M:xypqpz},
and $(xy'z)_M$ for~\eqref{d.M:xqyqpz}.
Let us sketch the verifications.

The resolvability of~\eqref{d.M:xyz} follows from the
fact that $M$ is an $\!R\!$-module.

The case of \eqref{d.M:xqpyz} is like that of~\eqref{d.xpqpyz},
the one difference being that
where there we wrote the leftmost basis element as $xp,$
here it is a general element $x\in M_+;$ but in either case,
our reduction~\eqref{d.M:xqpy|->}
allows us (roughly speaking) to drop a following ``$qp$''.

In~\eqref{d.M:xypqpz}, if we begin by reducing $x\cdot (yp)$
using~\eqref{d.M:xy|->}, that product becomes
$(xyp)_M,$ the representation of a member of $Mp\subseteq M_+,$
hence its expression in terms of $C$ involves only members of $C_+.$
Hence by~\eqref{d.M:xqpy|->}, when we multiply it by $q\,(pz),$
each of the resulting products reduces to the value we
would have gotten if
we had simply multiplied by $z,$ so the result is indeed $(xypz)_M.$
If instead we first reduce $(yp)\,q\cdot (pz)$ to $(ypz)_R$
using~\eqref{d.xpqpy|->}, then multiply $x$ by this,
applying~\eqref{d.M:xy|->} to each term
occurring, we get the same result $(xypz)_M.$

The calculation for~\eqref{d.M:xqyqpz} combines the features of the
two preceding cases.
The reduction of $xq\cdot y$ works as in the case of~\eqref{d.M:xqpyz}
once we rewrite $y$ as $py',$ and gives $(xy')_M.$
Moreover, because $py'=y''p,$ the subspace $M_+\subseteq M$ is carried
into itself by multiplication by $y'$ (see end of second paragraph of
Definition~\ref{D.tempered}), so $xy'\in M_+;$ hence multiplication of
$(xy')_M$ by $q\,(pz)$ is the same as multiplication
by $z,$ and gives $(xy'z)_M.$
On the other hand, if we begin by reducing
$y\cdot q\,(pz)=(y''p)\,q\,(pz)$
to $(y''pz)_R=(py'z)_R,$ then the result of multiplying $x\,q$ by
this is again $(xy'z)_M,$ by application of~\eqref{d.M:xqpy|->} to each
term occurring.
\end{proof}

Here are some easy consequences.

\begin{corollary}\label{C.M->MOX}
For $R,$ $p,$ $R'$ and $(M,M_+)$ as in Proposition~\ref{P.M_norm}, the
canonical $\!R\!$-module homomorphism $M\to(M,M_+)\otimes_{(R,p)} R'$
is an embedding; and identifying $M$ with its image under
this map, we have
\begin{equation}\begin{minipage}[c]{35pc}\label{d.M+=}
$M_+\ =\ M\ \cap\ ((M,M_+)\otimes_{(R,p)} R')\,p\ =
\ \{x\in M\mid x(qp-1)=0$ in $(M,M_+)\otimes_{(R,p)} R'\}.$
\end{minipage}\end{equation}

In particular, for any $\!p\!$-tempered $\!R\!$-module $(M,M_+),$
the module $M$ can be embedded in an $\!R\!$-module $N$ so that
$M_+=M\cap Np.$
\end{corollary}

\begin{proof}
The map $M\to(M,M_+)\otimes_{(R,p)} R'$ takes elements of
the $\!k\!$-basis $C$ of $M$ to themselves as elements
of the $\!k\!$-basis of $(M,M_+)\otimes_{(R,p)} R'$
described in Proposition~\ref{P.M_norm}; hence it is one-to-one.

In~\eqref{d.M+=}, it is easy to see that the leftmost
and rightmost subspaces are
equal, since for a $\!k\!$-linear combination $x$ of the elements
of $C,$ the reduction rules reduce $xqp$ to $x$ if and
only if all the basis elements occurring in $x$ belong to $C_+,$
i.e., if and only if $x\in M_+.$
To see the equality of the middle and rightmost subspaces, note
that in any right $\!R'\!$-module, and so in particular, in
$(M,M_+)\otimes_{(R,p)} R',$ every right multiple of $p$ is
annihilated by $qp-1,$ and conversely, any element $x$
satisfying $x(qp-1)=0$ satisfies $x=xqp,$ and so is a right
multiple of~$p.$

The final statement is seen on taking $N=(M,M_+)\otimes_{(R,p)} R',$
regarded as an $\!R\!$-module.
\end{proof}

\begin{corollary}\label{C.MinN}
Let $R,$ $p$ and $R'$ be as in Proposition~\ref{P.M_norm},
and let $h:(M,M_+)\to(N,N_+)$ be a morphism of $\!p\!$-tempered
$\!R\!$-modules.
Then the induced homomorphism of $\!R'\!$-modules $h\otimes_{(R,p)}R':
(M,M_+)\otimes_{(R,p)} R'\to (N,N_+)\otimes_{(R,p)} R'$
is one-to-one if and only if $h$ is
an embedding of $\!p\!$-tempered $\!R\!$-modules
in the sense of Definition~\ref{D.tempered}.
\end{corollary}

\begin{proof}
Suppose $h$ is an embedding of $\!p\!$-tempered $\!R\!$-modules.
Then without loss of generality, we can assume that $M$ is
a submodule of $N,$ and $M_+=M\cap N_+.$
Let us take a $\!k\!$-basis $C^{(0)}_+\cup C^{(0)}_-$ of $M$ as in
the statement of Proposition~\ref{P.M_norm}, and extend $C^{(0)}_+$
to a $\!k\!$-basis $C^{(0)}_+\cup C^{(1)}_+$ of $N_+.$
By Lemma~\ref{L.V_1+V_2}, $C^{(0)}_+\cup C^{(0)}_- \cup C^{(1)}_+$ is
a $\!k\!$-basis of $M+N_+,$ and we can extend this to a $\!k\!$-basis
$C^{(0)}_+\cup C^{(0)}_- \cup C^{(1)}_+\cup C^{(1)}_-$ of $N.$
If we now write this basis as
$(C^{(0)}_+\cup C^{(1)}_+)\cup(C^{(0)}_-\cup C^{(1)}_-)$ and
use it to define a normal form in $(N,N_+)\otimes_{(R,p)} R',$
we see that $(M,M_+)\otimes_{(R,p)} R'$ forms a submodule thereof;
so the induced homomorphism is one-to-one.

Conversely, if that induced homomorphism is one-to-one,
then restricting it to the embedded copies of $M$ and $N$
in those modules, we see that $h$ is one-to-one.
Moreover, the elements of $M$ that are annihilated by $qp-1$
in $(M,M_+)\otimes_{(R,p)} R'$ will be those whose images
are annihilated by that element in $(N,N_+)\otimes_{(R,p)} R',$
i.e., $M_+=h^{-1}(N_+).$
Thus the homomorphism is
indeed an embedding of $\!p\!$-tempered $\!R\!$-modules.
\end{proof}

\section{Do we need to go beyond the case $1\notin pR+Rp$?}\label{S.enough?}

Above, we have studied the properties of
$R'=R\lang q\mid p=pqp\rang$ when $1\notin pR+Rp.$
In the next five sections we examine cases where $1\in pR+Rp.$
But it may well be that for attacking the problem of
whether the monoid of finitely generated projectives of
a von~Neumann regular $\!k\!$-algebra is always separative,
the case considered above is all that matters.

Indeed, Pere Ara (personal communication) notes that the
separativity question for unital von~Neumann regular algebras
is equivalent to the same question for nonunital
von~Neumann regular algebras.
For if $R$ were a unital example with non-separative monoid, then
regarding it as a nonunital algebra, its (slightly larger) monoid of
projectives would still have that property; while conversely,
if we had a nonunital example $R,$ then the algebra
$R^1$ gotten by adjoining a unit to $R$ would be a unital example.
Note, moreover, that if $R$ is any nonunital $\!k\!$-algebra, then the
process of adjoining a universal inner inverse to an element $p\in R$
can be carried out by passing to $R^1,$ universally adjoining an inner
inverse to $p$ in $R^1$ as a unital algebra,
then dropping the adjoined unit (i.e., passing to the
nonunital subalgebra generated by $R\cup\{q\}).$
In this construction, $1\notin pR^1 + R^1p,$
since $pR^1 + R^1p\subseteq R;$ hence the construction of
universally adjoining an inner inverse to $p$ in $R^1$
falls under the case considered in the preceding sections.

Kevin O'Meara (personal communication) has likewise suggested
that the study of the separativity
question can be reduced to the case $1\notin pR+Rp.$

So the reader mainly interested in tackling that question using
our normal form may wish to skip or
skim~\S\S\ref{S.1-sided}-\S\ref{S.Weyl}.

However, the cases considered in those
sections seem interesting; especially
the case $1\in pR+Rp-(pR\cup Rp),$ where the elaborate complexity of
the normal form we shall discover suggests some strange
territory to be explored; so we include them.

Let us first get the easy case out of the way.

\section{Normal forms when $1\in pR$ and/or $1\in Rp.$}\label{S.1-sided}

Since the two cases $1\in pR$ and $1\in Rp$ are left-right dual,
let us assume without loss of generality that $1\in pR.$
This says $p$ has a right inverse; let us
fix such a right inverse $q_0\in R.$
It will clearly be an inner inverse to $p,$
so our motivation for adjoining
a universal inner inverse (to move our ring a step toward being
von~Neumann regular) is not relevant here; but for the sake of
our general understanding of the adjunction of inner inverses,
we are including this case.

If $q$ is any other inner inverse to $p,$ then right
multiplying the relation $pqp=p$ by $q_0,$ we get $pq=1;$
in other words, once $p$ has a right inverse, every
inner inverse to $p$ is a right inverse.
Moreover, subtracting the equations $pq_0=1$ and $pq=1,$
we get $p(q_0-q)=0;$ so all right inverses to $p$ are obtained
by adding to $q_0$ arbitrary elements that right annihilate~$p.$

Note that if both $1\in pR$ and $1\in Rp$ hold, then $p$
will be invertible, and if we adjoin a universal inner inverse,
it will have to be an inverse to $p,$ hence will fall together
with the existing inverse; so in that case, the adjunction
of a universal inner inverse to $p$ leaves $R$ unchanged.
Hence we will assume below that $1\in pR$ but $1\notin Rp.$
(In particular, $1\neq 0,$ equivalently, $R\neq\{0\}.)$

Since $1\in pR,$ we have $pR=R,$ so the analog of the sort of
basis of $R$ that we used in the preceding sections becomes simpler.
Namely, we take a basis
\begin{equation}\begin{minipage}[c]{35pc}\label{d.RI:B=}
$B\cup\{1\}\ =\ B_{++}\cup B_{+-}\cup \{1\},$
\end{minipage}\end{equation}
where
\begin{equation}\begin{minipage}[c]{35pc}\label{d.RI:B_}
$B_{++}$ is any $\!k\!$-basis of $Rp=pRp$ containing $p,$

$B_{+-}$ is any $\!k\!$-basis of $R=pR$ relative to $Rp+k.$
\end{minipage}\end{equation}

Our extension
\begin{equation}\begin{minipage}[c]{35pc}\label{d.RI:R'}
$R'\ =\ R\lang q\mid pqp=p\rang\ =\ R\lang q\mid pq=1\rang$
\end{minipage}\end{equation}
is clearly spanned by words in $B\cup\{q\}$ which contain no
subwords either of the form
\begin{equation}\begin{minipage}[c]{35pc}\label{d.RI:xy}
$xy$ \quad with $x,y\in B$
\end{minipage}\end{equation}
or of the form
\begin{equation}\begin{minipage}[c]{35pc}\label{d.RI:xpq}
$(xp)\,q$ \quad with $xp\in B_{++},$
\end{minipage}\end{equation}
and any expression in our
generators can be reduced to a linear combination of
such words via the system of reductions
\begin{equation}\begin{minipage}[c]{35pc}\label{d.RI:xy|->}
$xy\ \mapsto\ (xy)_R$ \quad for all $x,y\in B$
\end{minipage}\end{equation}
and
\begin{equation}\begin{minipage}[c]{35pc}\label{d.RI:xpq|->}
$(xp)\,q\ \mapsto\ x_R$ \quad for all $xp\in B_{++}\,.$
\end{minipage}\end{equation}
In contrast to the situation of the preceding sections,
the element $xp$ of~\eqref{d.RI:xpq|->}
{\em does} determine $x:$ if $q_0\in R$ is a right inverse to $p,$
we see that $x=(xp)\,q_0;$ so the expression $x_R$
in~\eqref{d.RI:xpq|->} is well-defined.

We find that the only ambiguities of this reduction system are
\begin{equation}\begin{minipage}[c]{35pc}\label{d.RI:xyz}
$x\cdot y\cdot z,$ where $x,y,z\in B,$
\end{minipage}\end{equation}
\begin{equation}\begin{minipage}[c]{35pc}\label{d.RI:xypq}
$x\cdot (yp) \cdot q,$ where
$x\in B,$ $yp\in B_{++},$
\end{minipage}\end{equation}
and it is straightforward to verify, by the approach used
in~\S\ref{S.norm}, that these are both resolvable.

Note that in the resulting normal form, elements of $B_{++}$ can
appear nowhere but in the last position in a reduced word.

(We would get the same ring $R'$ if we adjoined
to $R$ an element $u$ subject to the relation $pu=0;$
that extension is isomorphic to the one constructed above via
the identification of $q$ with $q_0+u.$
The construction using $u$ would be simpler to study on
its own, but the construction using $q$ lends itself better
to comparison with the other cases.)

We can likewise look at extension of scalars
from $\!R\!$-modules to $\!R'\!$-modules.
Since our assumption that $p$ has a right inverse is not
left-right symmetric, right and left modules
need to be considered separately.

If $M$ is a right $\!R\!$-module, we take, as in the preceding
section, a $\!k\!$-basis $C_+\cup C_-$ for $M,$ where $C_+$ is a
$\!k\!$-basis for $Mp.$
In this situation, we do not have to think about a more general
$\!k\!$-subspace $M_+,$ consisting of elements that might
become the multiples of $p$ in an overmodule, because the upper
and lower bounds for such an $M_+$ given
in Definition~\ref{D.tempered} coincide:
$x$ is a multiple of $p$ in $M$
if and only if $x=x q_0 p,$ i.e., if and only
if $x$ is annihilated by the element $1-q_0 p$ of the right annihilator
of $p$ in $R.$

It is straightforward to
verify that $M\otimes_R R'$ is spanned by words in $C\cup B\cup \{q\}$
in which elements of $C$ occur in the leftmost position and only there,
and which are irreducible under the reductions~\eqref{d.RI:xy|->}
and~\eqref{d.RI:xpq|->}, and also the corresponding reductions in
which the leftmost element of $B,$ respectively $B_{++},$
is replaced by an element of $C,$ respectively $C_+.$
In this case, we see that if the leftmost factor of
a reduced word is in $C_+,$ then that factor is the whole word.

Turning to left $\!R\!$-modules $M,$ we find that we do not
have to distinguish a subspace $pM$ or $M_+$ at all, since $pM=M.$
Again, we get reduced words having the same formal descriptions as
for reduced words of $R'.$
In this case, the analog of the fact that elements of $B_{++}$
and $C_+$ can only appear in final position is that elements of $B_{++}$
never appear.
(Whatever such an element might be followed by -- a member of $B,$
a $q,$ or a member of $C$ -- leads to a reducible word.)

Returning to our development of the structure of
$R'=R\lang q\mid pqp=p\rang,$
we remark that the uninteresting case that we referred to
briefly at the start of this section and then put aside,
where $p$ is invertible in $R,$ so that $R'=R,$
is the one case where the subalgebra of $R'$ generated by $q$
may fail to be a polynomial ring $k[q].$
Rather, it will, necessarily, fall together with
$k[p^{-1}]\subseteq R,$ which, if $p$ is algebraic over~$k,$
is the finite-dimensional subalgebra of $R$ generated by $p.$

\section{The case where $1\in pR+Rp - (pR\cup Rp):$ groping toward a normal form.}\label{S.1=}

We now consider the most difficult case, that in which
$1\in pR+Rp,$ but where $1$ does not lie in $pR$ or $Rp.$
In this section we illustrate the process of trying to find
a normal form, discovering more and more reductions as we go.
In the next section, we shall make precise the pattern that these
show, and prove that the resulting set of reductions does
lead to a normal form for $R'.$

Let us begin with a general observation, and a slight digression.

In any ring $R$ with an element $p$ that
has an inner inverse $q,$ so that
\begin{equation}\begin{minipage}[c]{35pc}\label{d.1=:pqp}
$pqp\ =\ p,$
\end{minipage}\end{equation}
note that
$p$ is left-annihilated by $pq-1$ and right-annihilated by $qp-1.$
Consequently,
\begin{equation}\begin{minipage}[c]{35pc}\label{d.1=:pq-1*qp-1}
$(pq-1)(pR+Rp)(qp-1)\ =\ \{0\}.$
\end{minipage}\end{equation}

Hence in the situation we are now interested in,
where $1\in pR+Rp,$ we get $(pq-1)1(qp-1)\ =\ 0,$ i.e.,
\begin{equation}\begin{minipage}[c]{35pc}\label{d.1=:pqqp}
$pqqp\ =\ pq + qp - 1.$
\end{minipage}\end{equation}

In the $\!k\!$-algebra $R$ presented
simply by two generators $p$ and $q$ and the
relations~\eqref{d.1=:pqp} and~\eqref{d.1=:pqqp}, we
find that the reduction system
\begin{equation}\begin{minipage}[c]{35pc}\label{d.1=:pqp|->}
$pqp\ \mapsto\ p,$
\end{minipage}\end{equation}
\begin{equation}\begin{minipage}[c]{35pc}\label{d.1=:pqqp|->}
$pqqp\ \mapsto\ pq+qp-1$
\end{minipage}\end{equation}
satisfies the conditions of the Diamond Lemma: there are just four
ambiguities, corresponding to the words
\begin{equation}\begin{minipage}[c]{35pc}\label{d.1=:pq_only}
$pq\cdot p\cdot qp,$ \quad
$pq\cdot p\cdot qqp,$ \quad
$pqq\cdot p\cdot qp,$ \quad
$pqq\cdot p\cdot qqp,$
\end{minipage}\end{equation}
and straightforward computations show that these are all resolvable.
So this algebra has a normal form with basis the set of all strings of
$\!p\!$'s and $\!q\!$'s that contain no substrings $pqp$ or $pqqp.$
Curiously, this algebra itself satisfies $1\in pR+Rp,$
by~\eqref{d.1=:pqqp}.
Consequently, it is universal among $\!k\!$-algebras $R$
given with elements $p$ and $q$ satisfying~\eqref{d.1=:pqp}
and such that $1\in pR+Rp.$
(It is not, however, universal among $\!k\!$-algebras given
with elements $p$ and $q$ satisfying~\eqref{d.1=:pqp}
together with specified elements $s$ and $t$ such that $1=ps+tp,$
i.e., it is not $k\lang p, q, s, t\mid p=pqp,\,1=ps+tp\rang,$
the universal example one would first think of.)

The above algebra might be worthy of study,
but it is not one of the algebras we are preparing to investigate here.
Those are the algebras $R'$ gotten by starting with a $\!k\!$-algebra
$R$ given with an element $p$ such that
\begin{equation}\begin{minipage}[c]{35pc}\label{d.1=_but}
$1\in pR+Rp$ but $1\notin pR,$ $1\notin Rp,$
\end{minipage}\end{equation}
and adjoining a universal inner inverse $q$ to $p.$

As in the preceding sections, we shall start by taking an
appropriate $\!k\!$-basis of $R.$
A problem is that since $1\in pR+Rp,$ we can't take a $\!k\!$-basis
containing sets $B_{++},$ $B_{+-}$ and $B_{-+}$ as
in~\eqref{d.B=} and~\eqref{d.B_},
and {\em also} the unit $1$ (which we want to represent
by the empty word in our normal form for $R').$
What we shall do instead is choose a spanning set for $R$
rather like that of~\eqref{d.B=} and~\eqref{d.B_}, but
which is not quite $\!k\!$-linearly independent,
then handle the one linear relation it satisfies as
an extra reduction,~\eqref{d.1=:ps|->} below.

(Naively we might, instead, think of using a normal form for
$R'$ in which $1$ is not represented by the empty monomial, but by the
sum of a basis element from $pR$ and a basis element from $Rp.$
However, the version of the Diamond Lemma we are using requires that
we regard $1$ as the empty monomial in our generators,
so that won't work.
It {\em would} work if we use the version of the Diamond Lemma for
nonunital rings.
But then, to restore unitality, we would have
to throw in reductions that force our new generator $q$ to be fixed
under left and right multiplication by the sum-of-generators that gives
the $1$ of $R,$ and this seems messier than the path we shall follow.)

So given $R$ satisfying~\eqref{d.1=_but},
let us fix elements $s,$ $t\in R$ such that
\begin{equation}\begin{minipage}[c]{35pc}\label{d.1=}
$1\ =\ ps+tp$ in $R,$
\end{minipage}\end{equation}
and choose a spanning set for $R$ as a $\!k\!$-vector-space, of the form
\begin{equation}\begin{minipage}[c]{35pc}\label{d.1=:B=}
$B\cup\{1\}\ =\ B_{++}\cup B_{+-} \cup B_{-+}\cup B_{--}\cup \{1\},$
\end{minipage}\end{equation}
where
\begin{equation}\begin{minipage}[c]{35pc}\label{d.1=:B_}
$B_{++}$ is any $\!k\!$-basis of $pR\cap Rp$ containing
the element $p,$

$B_{+-}$ is any $\!k\!$-basis of $pR$ relative to $pR\cap Rp$
containing the element $ps,$

$B_{-+}$ is any $\!k\!$-basis of $Rp$ relative to $pR\cap Rp$
containing the element $tp,$

$B_{--}$ is any $\!k\!$-basis of $R$ relative to $pR+Rp.$
\end{minipage}\end{equation}
Note that the condition above that $B_{+-}$ contain $ps$ can
be achieved because
$ps$ does not lie in $pR\cap Rp;$ if it did,~\eqref{d.1=}
would imply $1\in Rp,$ contrary to our assumptions.
By the dual observation, the condition
that $B_{-+}$ contain $tp$ can also be achieved.
Since $pR+Rp$ contains $1,$ we don't have to
throw a ``$+k$'' onto the $pR+Rp$ in the description of $B_{--}$
as in~\eqref{d.B_}.
Thus the above $B$ will be a $\!k\!$-basis
of $R$ by Lemma~\ref{L.V_1+V_2}.
But this means that $B\cup\{1\}$ will not.
Rather, by~\eqref{d.1=}, $B\cup\{1\}-\{ps\}$ will be a $\!k\!$-basis
of $R.$

For any $\!k\!$-algebra expression $f$ in the elements of $B,$
let $f_R$ denote the unique $\!k\!$-linear combination
of elements of $B\cup\{1\}-\{ps\}$ representing the value
of $f$ in $R.$
Then we see that $R$ can be presented using the generating set $B,$
the relations corresponding to the reductions
\begin{equation}\begin{minipage}[c]{35pc}\label{d.1=:xy|->}
$xy\ \mapsto\ (xy)_R$ \quad for all $x,y\in B,$
\end{minipage}\end{equation}
and the relation corresponding to the single additional reduction
\begin{equation}\begin{minipage}[c]{35pc}\label{d.1=:ps|->}
$(ps)\ \mapsto\ 1-(tp).$
\end{minipage}\end{equation}

We now construct $R'$ by adjoining an additional generator $q,$
and imposing the relation $pqp=p.$
As in \S\ref{S.norm}, this leads to the further reductions
\begin{equation}\begin{minipage}[c]{35pc}\label{d.1=:xpqpy|->}
$(xp)\,q\,(py)\ \mapsto\ (xpy)_R$ \quad for all
$xp\in B_{++}\cup B_{-+},$ $py\in B_{++}\cup B_{+-}\,.$
\end{minipage}\end{equation}

But in view of~\eqref{d.1=:pqqp}, these reductions
cannot be sufficient to give a normal form for $R',$ so they must have
non-resolvable ambiguities.

And indeed, note that
for any $xp\in B_{++}\cup B_{-+},$ the word $(xp)\,q\,(ps)$
is ambiguously reducible, using~\eqref{d.1=:xpqpy|->} on the
one hand or~\eqref{d.1=:ps|->} on the other.
Equating the results gives the
relation $(xps)_R = (xp)\,q-(xp)\,q\,(tp).$
Regarding this as a formula for reducing the
longest monomial that it involves, $(xp)\,q\,(tp),$
we get a new family of reductions,
\begin{equation}\begin{minipage}[c]{35pc}\label{d.1=:xpqtp|->}
$(xp)\,q\,(tp)\ \mapsto\ (xp)\,q-(xps)_R$\quad
for all $xp\in B_{++}\cup B_{-+}\,.$
\end{minipage}\end{equation}

These, in turn, lead to an ambiguity in the reduction of
any word $(xp)\,q\cdot(tp)\cdot q\,(py):$
we can apply~\eqref{d.1=:xpqtp|->}, getting
$(xp)\,qq\,(py)-(xps)_R\,q\,(py),$
or~\eqref{d.1=:xpqpy|->}, getting $(xp)\,q\,(tpy)_R.$
So let us again make the relation equating these
expressions into a reduction affecting the longest word
occurring, which is now $(xp)\,qq\,(py).$
With a view to what is to come, I will number this
\begin{equation}\tag{\xppy{2}}\begin{minipage}[c]{35pc}
$(xp)\,qq\,(py)\ \mapsto\ (xps)_R\,q\,(py)+(xp)\,q\,(tpy)_R$\quad
for all $xp\in B_{++}\cap B_{-+}$ and $py\in B_{++}\cap B_{+-}\,.$
\end{minipage}\end{equation}

(Digression: If we rewrite the factor $(xps)_R$ in the first term
of the output of the above
reduction as $(x(1\,{-}\,tp))_R=x_R-(xtp)_R,$ and inversely,
rewrite the factor $(tpy)_R$ at the end of the last term as
$((1\,{-}\,ps)y)_R=y_R-(psy)_R,$ then
the resulting terms of~(\xppy{2}) include
$(xtp)_R\,q\,(py)$ and $(xp)\,q\,(psy)_R,$ which
by~\eqref{d.1=:xpqpy|->} reduce to $(xtpy)_R$ and $(xpsy)_R,$
which then sum to $(x(ps+tp)y)_R=(xy)_R.$
This turns~(\xppy{2}) into
\begin{equation}\begin{minipage}[c]{35pc}\label{d.1=:xpqqpyalt}
$(xp)\,qq\,(py)\ \mapsto\ x_R\,q\,(py)+(xp)\,q\,y_R-(xy)_R.$
\end{minipage}\end{equation}
This can be seen as embodying~\eqref{d.1=:pqqp}; it represents
the result of multiplying that equation on the
left by $x$ and on the right by $y.$
The form~\eqref{d.1=:xpqqpyalt} has the nice feature of
not depending on the choice of $s$ and $t$
in~\eqref{d.1=}, but it has the downside
that it involves expressions $x_R,$ $y_R$ and $(xy)_R$ which
are not uniquely determined by the given basis elements
$xp$ and $py,$ in contrast to the situation we had in \S\ref{S.norm},
where expressions occurring in our reductions, such as $(xpy)_R,$
were shown to depend only on the basis elements $xp$ and $py.$
For this reason we will
use~(\xppy{2}) rather than~\eqref{d.1=:xpqqpyalt}.)

The five families of reductions~\eqref{d.1=:xy|->}, \eqref{d.1=:ps|->},
\eqref{d.1=:xpqpy|->}, \eqref{d.1=:xpqtp|->},
(\xppy{2}) that we have accumulated
at this point admit $20$ families of ambiguities!
Namely, the {\em final} factor in $B$ of the
input-monomials of each of these
sorts of reductions can coincide with the {\em initial} factor in $B$
of the input-monomials of most of these sorts of reduction,
the exceptions being that the
lone factor $(ps)$ of the input
of~\eqref{d.1=:ps|->} cannot coincide with the
initial factors of the inputs of~\eqref{d.1=:xpqpy|->},
\eqref{d.1=:xpqtp|->} or (\xppy{2}),
nor with the final factor of the input of~\eqref{d.1=:xpqtp|->}
(thus eliminating four of the $25$ potential pairings);
nor does one get an ambiguity by overlapping~\eqref{d.1=:ps|->}
with itself.

Many of these $20$ sorts of
ambiguities are already resolvable, either because
of the way they incorporate the structure of the associative ring $R,$
or because some of the later reductions were introduced precisely to
make earlier ambiguities resolvable.
Summarizing long and tedious hand computations (which we will be able
to circumvent in the next section), one finds that
of those $20$ sorts of ambiguities, 17 are
resolvable, the three exceptions being
\begin{equation}\begin{minipage}[c]{35pc}\label{d.three.ambig}
$(xp)\,q\cdot (tp)\cdot q\,(tp),\qquad
(xp)\,q\cdot (tp)\cdot qq\,(py),\qquad
(xp)\,qq\cdot(ps).$
\end{minipage}\end{equation}

Of these, the first and third turn out to yield a common relation.
Selecting, as usual, the longest monomial in that relation,
and writing the result as a formula reducing
that monomial to a combination of the others, this takes the form
\begin{equation}\tag{\xptp{2\,}}\begin{minipage}[c]{35pc}
$(xp)\,qq\,(tp)\ \mapsto\ (xp)\,qq - (xps)_R\,q + (xps)_R\,q\,(tp)
- (xp)\,q\,(tps)_R.$
\end{minipage}\end{equation}

The middle ambiguity shown in~\eqref{d.three.ambig}
yields a different reduction:
\begin{equation}\tag{\xppy{3}}\begin{minipage}[c]{35pc}
$(xp)\,qqq\,(py)\ \mapsto\ (xps)_R\,qq\,(py) + (xp)\,q\,(tps)_R\,q\,(py)
+ (xp)\,qq\,(tpy)_R - (xps)_R\,q\,(tpy)_R.$
\end{minipage}\end{equation}

Examining the reductions we have been getting (after~\eqref{d.1=:xy|->}
and~\eqref{d.1=:ps|->}, which describe $R$ itself), they
appear to fall into two series (as indicated in
numbering I have given them),
one starting with~\eqref{d.1=:xpqpy|->}, (\xppy{2}),
(\xppy{3}), the other with~\eqref{d.1=:xpqtp|->}, (\xptp{2}).
Examining which ambiguities turned out to yield which new
reductions, one can guess which should yield the
next term in each series.
In this way one finds, for instance, the next reduction
in the \eqref{d.1=:xpqpy|->}-series:
\begin{equation}\tag{\xppy{4}}\begin{minipage}[c]{35pc}
$(xp)\,qqqq\,(py)\ \mapsto\ (xps)_R\,qqq\,(py) +
(xp)\,q\,(tps)_R\,qq\,(py) + (xp)\,qq\,(tps)_R\,q\,(py)
+ (xp)\,qqq\,(tpy)_R\\
\hspace*{2em}- (xps)_R\,q\,(tps)_R\,q\,(py)
- (xps)_R\,qq\,(tpy)_R - (xp)\,q\,(tps)_R\,q\,(tpy)_R.$
\end{minipage}\end{equation}

The pattern of the inputs of~\eqref{d.1=:xpqpy|->}, (\xppy{2}),
(\xppy{3}), (\xppy{4}) is clear; but what about the outputs?
It appears that (ignoring signs, for the
moment), each term in the output of a reduction~``(\xppy{n})'' is
obtained from the input monomial $(xp)\,q^n\,(py)$ by
replacing or not replacing the initial $(xp)\,q$ with $(xps)_R,$
replacing or not replacing the final $q\,(py)$ with $(tpy)_R,$ and
replacing or not replacing some of the remaining $\!q\!$'s
with $(tps)_R.$
But not every possible combination
of such changes and non-changes shows up in our reductions; only
those where no two successive $\!q\!$'s are changed.

Can we make sense of this?

\section{The normal form, described and proved.}\label{S.1=:norm}

The relations in $R'$ that yield the reductions~(\xppy{n})
can in fact be derived from scratch in roughly the way we obtained the
relation~\eqref{d.1=:pqqp}; namely, by inserting terms $1=ps+tp$
between certain factors of the input
monomial, partly expanding the result, and then
simplifying using~\eqref{d.1=} and~\eqref{d.1=:xpqpy|->}.
For example, the relation corresponding to~(\xppy{2}) can
be gotten as follows:
\begin{equation}\begin{minipage}[c]{25pc}\label{d.insert-2}
\hspace*{-5.8em}$(xp)\,qq\,(py)\ =
\ (xp)\,q\,(ps+tp)\,q\,(py)\ \\[.3em]
=\ (xp)\,q\,(ps)\,q\,(py)+ (xp)\,q\,(tp)\,q\,(py)\ \\[.3em]
=\ (xps)_R\,q\,(py)+(xp)\,q\,(tpy)_R.$
\end{minipage}\end{equation}

However, not every string of insertions of terms $(ps)$
and $(tp)$ between $\!q\!$'s in a word $(xp)\,q\dots q\,(py)$
admits a simplification of the sort used above.
We could not, for instance, simplify
a string $\dots (tp)\,q\,(tp)\dots$
using~\eqref{d.1=:xpqpy|->}, because the second $(tp)$ does not
begin with a $p,$ and so does not give us a ``$pqp$'' to reduce.

So the equations on which we should perform the simplifications that
will yield the reductions~(\xppy{n}) for general $n$ are not obvious.
For instance, the next case,~(\xppy{3}),
can be obtained similarly by writing $(xp)\,qqq\,(yp)$
as $(xp)\,q\,(ps+tp)\,q\,(ps+tp)\,q\,(py),$ then using the expansion
\begin{equation}\begin{minipage}[c]{35pc}\label{d.insert-3}
$(xp)\,q\,(ps+tp)\,q\,(ps+tp)\,q\,(py)\ =
\ (xp)\,q\,(ps)\,q\,(ps+tp)\,q\,(py)\\[.3em]
\hspace*{2em}+ (xp)\,q\,(tp)\,q\,(ps)\,q\,(py)
+ (xp)\,q\,(ps+tp)\,q\,(tp)\,q\,(py)
- (xp)\,q\,(ps)\,q\,(tp)\,q\,(py),$
\end{minipage}\end{equation}
and reducing these terms.
That~\eqref{d.insert-3} is an identity of associative
rings is easy to check.
(Clearly, before checking it we can drop the initial $(xp)$ and
final $(py)$ of each term.)
But it is not obvious how we would have come up
with that identity to use.
In the next lemma we shall describe and prove
a sequence of identities to which the result of deleting
the initial $(xp)$ and final $(py)$ from each term
of~\eqref{d.insert-3} belongs, and the $\!n\!$-th step of which
will similarly allow us to obtain~(\xppy{n}).

(Though I believe in the principle of stating results in their
abstractly most natural form, since they may prove useful in
contexts very different from the ones for which they were
devised, the lemma below is unabashedly rigged to be used in
the specific context we will apply it in, for the sake of smoothing
that application.
I will re-state it in a more general form as
Corollary~\ref{C.multilin},
when we are through with the work of this section.)

\begin{lemma}\label{L.identity}
Let $n\geq 2$ be an integer, $F$ the free associative
$\!k\!$-algebra on generators $p,$ $s,$ $t,$ $q,$ and
$S(n)$ the set of elements of $F$ which can be obtained by the
following procedure:
\begin{equation}\begin{minipage}[c]{35pc}\label{d.insert}
Starting with the monomial $q^n,$ insert between
each pair of successive $\!q\!$'s either $(ps+tp),$ or $(ps)$
alone, or $(tp)$ alone, in such a way that every $q$ that is
immediately preceded by $(tp)$
is either immediately followed by $(ps)$ or is the final $q,$ and
every $q$ that is followed immediately by $(ps)$
is either immediately preceded by $(tp)$ or is the initial $q.$

\hspace*{1em}Then multiply the resulting element by
$(-1)^d,$ where $d$ is the number
of $\!q\!$'s in its expression which are preceded by
$(tp)$ and/or followed by $(ps).$
\textup{(}I.e., which are initial and followed by $(ps),$
or are simultaneously preceded by $(tp)$ and followed by $(ps),$
or are final and preceded by $(tp)).$
\end{minipage}\end{equation}

Then the sum in $F$ of the set $S(n)$ is $0.$
\end{lemma}

\begin{proof}
Let us multiply out each element of the set $S(n)$ described
above to get a sum of monomials;
i.e., wherever a factor $(ps+tp)$ occurs in such a product, write the
product as the sum of a product having $(ps)$ and
a product having $(tp)$ in that position.
Thus, each of the resulting monomials will contain
$n$ $\!q\!$'s, with every pair of successive $\!q\!$'s having
either a $(tp)$ or a $(ps)$ between them.
Let $W(n)$ be the set of all monomials of this form.
We must prove that for every $w\in W(n),$ the sum of the
coefficients with which it occurs in members of $S(n)$ is $0.$

Within a monomial $w\in W(n),$ let us call an
occurrence of $q$ ``marked'' if it is
initial and followed by $(ps),$
or preceded by $(tp)$ and followed by $(ps),$
or final and preceded by $(tp).$
Every $w\in W(n)$ has at least one marked $q;$ for if there
is at least one factor $(ps),$ then the $q$ preceding the first
such factor (whether it is initial or preceded by a $(tp))$
will be marked, while if there are no factors $(ps),$ then
the final $q$ will be preceded by a $(tp),$ and hence marked.
On the other hand, two successive $\!q\!$'s can never both be marked,
since if they have a $(tp)$ between them, the left-hand $q$ won't
be marked, while if they have a $(ps)$ between them, the right-hand $q$
won't be marked.

Let $e\geq 1$ be the number of marked $\!q\!$'s in $w.$
I claim that there are exactly $2^e$ elements $v\in S(n)$ which contain
a $\pm w$ in their expansion, half of them with a plus sign
and half with a minus sign.
Indeed, given $w,$ all such elements
$v\in S(n)$ can be found by a construction that makes
the following binary choice at each marked $q$ of $w:$
If the marked $q$ in question is neither initial nor final, so that
it is preceded by a $(tp)$ and followed by a $(ps),$ the choice
is between keeping these factors $(tp)$ and $(ps)$ unchanged in $v,$
or replacing both with $(ps+tp).$
(The definition of $S(n)$ doesn't allow any other possibilities.)
If the $q$ in question is initial, the choice is simply between
keeping the following $(ps)$ unchanged or replacing it with $(ps+tp),$
while if it is final, the choice is between
keeping the preceding $(tp)$ unchanged or replacing it with $(ps+tp).$
(Since successive $\!q\!$'s cannot be marked, the effects
of choices at different marked $\!q\!$'s
will not conflict with each other.)
Finally, for factors $(tp)$ of $w$ that do not precede
marked $\!q\!$'s, and factors $(ps)$ that do not follow
marked $\!q\!$'s, there is no choice: we replace these with $(ps+tp).$

It is not hard to see from the definition of $S(n)$
that these $2^e$ ways of modifying $w$ indeed give precisely
the elements $v\in S(n)$ that have $w$ in their expansion.
Moreover, by the second paragraph of~\eqref{d.insert},
such an element of  $S(n)$ will bear a plus sign if the
number of marked $\!q\!$'s around which we did not choose to change
the adjacent factor(s) of $w$ to $(ps+tp)$
is even, a minus sign if that number is odd.
Hence half of the resulting occurrences
of $w$ have a plus sign and half have a minus
sign, so they sum to zero; and since this is true for each $w,$ we
get $\sum_{v\in S(n)} v =0,$ as claimed.
\end{proof}

Now returning to the $\!k\!$-algebra
$R'=R\lang q\mid p=pqp\rang,$ where $1=ps+tp$ in $R,$
let us map the free algebra of
the above lemma into $R'$ by sending each indeterminate
to the element of $R'$ denoted by the same letter.
The lemma then tells us that in $R',$
a certain sum of signed products is zero.
Hence if we choose any $(xp)\in B_{++}\cup B_{-+}$
and $(py)\in B_{++}\cup B_{+-},$ and multiply that sum
of products on the left by $(xp)$ and on the right
by $(py),$ the resulting sum of products is still zero.
In the expressions for these products,
we now can cross out each factor $(ps+tp),$ since it equals $1,$
and replace occurrences of $(xp)\,q\,(ps),$
$(tp)\,q\,(ps),$ and $(tp)\,q\,(py)$ by $(xps)_R,$
$(tps)_R,$ and $(tpy)_R$ respectively.
The one element of
$(xp)\,S(n)\,(py)$ in which no reduction of these three sorts is made
is the one that had $(ps+tp)$ in all $n{-}1$ positions,
and is now simply $(xp)\,q^n\,(py).$
Using the relation we have
obtained to express that product as a linear combination of
products with fewer remaining $\!q\!$'s, we get,

\begin{corollary}\label{C.identity}
For $R,$ $p,$ $B$ as in the preceding section, any
$(xp)\in B_{++}\cup B_{-+}$ and $(py)\in B_{++}\cup B_{+-},$
and any $n\geq 2,$
let $T((xp),n,(py))$ be the set of $\!k\!$-linear combinations
of words in $B$ formed by
modifying the word $(xp)\,q^n\,(py)$ as follows:

Choose any {\em nonempty} subset of the string of $n$ $\!q\!$'s
in that word, to be called ``marked'' $\!q\!$'s, such that
no two adjacent $\!q\!$'s are both marked.
If the first $q$ in the string is marked, replace the initial
term $(xp)\,q$ with $(xps)_R.$
If the last $q$ is marked, replace the final
term $q\,(py)$ with $(tpy)_R.$
Replace every marked $q$ that is neither initial nor final
with $(tps)_R.$
Finally, multiply the result by $-1$ if the number of marked $\!q\!$'s
was~{\em even}.

Then the reduction
\begin{equation}\tag{\xppy{n}}\begin{minipage}[c]{35pc}
$(xp)\,q^n\,(py)\ \mapsto\ \sum_{v\in T((xp),n,(py))}\ v$
\end{minipage}\end{equation}
corresponds to a relation holding in $R'.$
\textup{(}I.e., the elements of $R'$ represented
by the input and the output of~\textup{(\xppy{n})}
are equal.\textup{)}\qed
\end{corollary}

Some remarks before we go further:

The number of terms in the set $S(n)$ of Lemma~\ref{L.identity}
(and hence in the reduction~(\xppy{n}), counting the
input term as well as the terms in $T((xp),n,(py))),$
is the $\!n{+}2\!$'nd
Fibonacci number, $F_{n+2},$ since this is known to
be the number of subsets of a sequence of
$n$ elements containing no two successive elements
\cite[p.\,14, Problem~1(b)]{comb}.

The reduction~\eqref{d.1=:xpqpy|->} clearly deserves to be
called~(\xppy{1}); but we assumed $n\geq 2$ in the preceding
lemma and corollary because the $n=1$ case differs from the
general case in a couple of ways.
On the one hand, when $n=1,$ the initial $q$ is also the final
$q,$ so we get an output term $(xpy)_R,$ which is
not one of the three sorts that occur when $n\geq 2.$
More important, the two sides of~\eqref{d.1=:xpqpy|->} do not
differ as a result of where factors $(ps+tp),$ $(ps)$ or $(tp)$
appeared in a term $v,$ but simply as to
whether or not one reduces the product $(xp)\,q\,(py)$
in the tautology $(xp)\,q\,(py)=(xp)\,q\,(py).$
Nevertheless,~\eqref{d.1=:xpqpy|->} has precisely the right form
to be described as reduction~(\xppy{1}), and we will so
consider it when we describe our normal form for~$R'.$
(We might consider the lone $q$ ``unmarked'' in the input
of~\eqref{d.1=:xpqpy|->} and ``marked'' in the output.)

Let us note, finally, that monomials occurring in the output
of~(\xppy{n}) may admit further reductions.
For instance, in the output term $(xps)_R\,q\,(py)$ of~(\xppy{2}),
some of the elements of $B$ appearing in $(xps)_R$ may
be of the form $(x'p),$
allowing reductions $(x'p)\,q\,(py)\mapsto (x'py)_R.$
(Indeed, all of them will have this form if $x$ is a right
multiple of $p$ in $R,$ since then $xps=x(1-tp)$ will be a right
multiple of $p.)$
However, this does not interfere with our application of
the Diamond Lemma.
The formulation of that lemma
does not require that the output of each reduction not
admit further reductions, but simply that it be a linear
combination of words smaller than
the word one started with, under an appropriate partial ordering.
\vspace{.5em}

We now turn to the other family of reductions we
encountered, beginning with~\eqref{d.1=:xpqtp|->}.
Since, as just noted, it is not essential that all
the terms of the outputs of our reductions
be, themselves, reduced, we can make a slight
simplification in the form of~(\xptp{2}),
replacing the final $(tp)$ in the
third output term by $(1-ps),$ to which it is equal in $R.$
Two terms then cancel, after which~(\xptp{2}) takes the form
\begin{equation}\tag{\xpTP{2}}\begin{minipage}[c]{35pc}
$(xp)\,qq\,(tp)\ \mapsto\ (xp)\,qq - (xps)_R\,q\,(ps)
- (xp)\,q\,(tps)_R.$
\end{minipage}\end{equation}

This leads to a version of the~\eqref{d.1=:xpqtp|->}-series
of reductions that is easily deduced from Corollary~\ref{C.identity}.

\begin{corollary}\label{C.identity_tp}
For every $n\geq 2$ and $(xp)\in B_{++}\cup B_{-+},$ the reduction
\begin{equation}\tag{\xpTP{n}}\begin{minipage}[c]{35pc}
$(xp)\,q^n\,(tp)\ \mapsto\ (xp)\,q^n - \sum_{v\in T((xp),n,(ps))} v,$
\end{minipage}\end{equation}
where $\sum_{v\in T((xp),n,(ps))} v$ is defined
as in Corollary~\ref{C.identity},
corresponds to a relation holding in $R'.$
\end{corollary}

\begin{proof}
Applying Corollary~\ref{C.identity} with $py=ps$
(which is allowable, since $ps\in B_{+-}),$ we get
$(xp)\,q^n\,(ps)=\sum_{v\in T((xp),n,(ps))} v$ in $R'.$
Rewriting the factor $(ps)$ on the left-hand
side as $1-(tp),$ multiplying out, moving the shorter of the
two resulting terms to the right-hand side, and changing all signs,
we get $(xp)\,q^n\,(tp)=(xp)\,q^n - \sum_{v\in T((xp),n,(ps))} v,$
the desired relation.
\end{proof}

We now have four families of reductions,~\eqref{d.1=:xy|->},
\eqref{d.1=:ps|->}, (\xppy{n}) and~(\xpTP{n}),
where in the last two, we from now on allow all $n\geq 1,$
counting~\eqref{d.1=:xpqpy|->} as (\xppy{1}),
and~\eqref{d.1=:xpqtp|->} as~(\xpTP{1});
and we wish to show that these together
determine a normal form for elements of $R'.$
We have established that they correspond to relations
holding in $R'.$
Moreover, they imply the defining relations for that
$\!k\!$-algebra in terms of our generating set $B\cup\{q\},$
since~\eqref{d.1=:xy|->} and~\eqref{d.1=:ps|->} determine
the structure of $R,$ while the imposed relation $pqp=p$
is the case of~(\xppy{1}) where
both $xp$ and $py$ are $p.$
It remains to find a partial ordering on words in $B\cup\{q\}$
respecting multiplication and having descending chain condition,
with respect to which all of these reductions are strictly decreasing,
and to prove that the ambiguities of the resulting reduction system
are resolvable.

The required partial ordering can be obtained by associating
to every word $w$ in $B\cup\{q\}$ the $\!3\!$-tuple with
first entry the number of $\!q\!$'s in $w,$
second entry the number of occurrences of members of $B$ in $w,$
and third entry the number of occurrences
of the particular element $(ps)\in B$ in $w,$ and considering one word
greater than another if the corresponding $\!3\!$-tuples
are so related under lexicographic order, while considering
distinct words which correspond to the same $\!3\!$-tuple incomparable.
It is easy to see that this ordering has descending chain
condition and respects formal
multiplication of words (juxtaposition), and that in
each of our reductions, all words of the output are strictly less
than the input word.
(The first coordinate of the $\!3\!$-tuple is enough to
show this last property for the reductions~(\xppy{n});
the second coordinate is needed for the reductions~\eqref{d.1=:xy|->},
and for the first term of the output of~(\xpTP{n}), while the
third coordinate is only needed for~\eqref{d.1=:ps|->}.)

Proving resolvability of ambiguities will, of course, be the hard task.

The ambiguities among
the cases of~\eqref{d.1=:xy|->} and~\eqref{d.1=:ps|->}
are, as usual, resolvable because they describe the
structure of the associative $\!k\!$-algebra~$R.$

I claim that ambiguities based on the fact that a word can be
reduced either by~(\xppy{n}) or by~\eqref{d.1=:xy|->}, i.e., those
involving words of the forms $(xp)\,q^n\,\cdot (py)\cdot z$ and
$x\cdot (yp)\cdot q^n\,(pz),$ are also easily shown to be resolvable.
As in the case of the ambiguities~\eqref{d.xpqpyz}
and~\eqref{d.xypqpz} considered in \S\ref{S.norm}, the
reason will be, in the former case, that right multiplication
by $z$ carries $pR$ left $\!R\!$-linearly into itself, and in the
latter, that left multiplication by $x$ carries $Rp$ into itself
right $\!R\!$-linearly;
so that whether we apply the reduction~(\xppy{n}) before
or after that operation, we get the same result.
For more detail, let us, in the case of $(xp)\,q^n\,\cdot (py)\cdot z,$
subdivide the summands in $\sum_{v\in T((xp),n,(py))} v$ in~(\xppy{n})
according to whether they end with $(py)$ or $(tpy)_R,$ writing
that reduction as
\begin{equation}\begin{minipage}[c]{35pc}\label{d.A,B}
$(xp)\,q^n\,(py)\ \mapsto
\ \sum_{v\in T((xp),n,(py))}\ v\ =
\ A((xp),n)\,(py) + B((xp),n)\,(tpy)_R.$
\end{minipage}\end{equation}
The reader can now easily verify that whichever of the
two competing reductions we perform first on
$(xp)\,q^n\,\cdot (py)\cdot z,$
the output can be reduced to $A((xp),n)(pyz)_R + B((xp),n)(tpyz)_R.$

The case of $x\cdot (yp)\cdot q^n\,(pz)$ is handled similarly, using
a decomposition of the elements of $T((xp),n,(py))$
by initial rather than final factors, which we
write down for later reference as
\begin{equation}\begin{minipage}[c]{35pc}\label{d.C,D}
$(xp)\,q^n\,(py)\ \mapsto\ \sum_{v\in T((xp),n,(py))}\ v\ =
\ (xp)\,C(n,(py)) + (xps)_R\,D(n,(py)),$
\end{minipage}\end{equation}
(though in the present application,
the roles of the $(xp)$ and $(py)$ in the above formula
are played by the elements $(yp)$ and $(pz)).$

The resolution of
ambiguities arising from words $x\cdot (yp)\cdot q^n\,(tp),$ which
can be reduced using either~\eqref{d.1=:xy|->} on the
left or~(\xpTP{n}) on the right, is verified similarly.

A little more complicated is the case of $(xp)\,q^n\cdot (tp)\cdot y,$
which can be reduced either by applying~(\xpTP{n}) on the left,
or~\eqref{d.1=:xy|->} on the right.
We recall that the result of reducing
$(xp)\,q^n(tp)$ by~(\xpTP{n})
is $(xp)\,q^n$ minus the result of reducing
$(xp)\,q^n(ps)$ by~(\xppy{n}); so applying this reduction in
$(xp)\,q^n (tp)\,y,$ and then making appropriate applications
of~\eqref{d.1=:xy|->}, we get $(xp)\,q^n y$
minus the result of reducing
$(xp)\,q^n(psy)_R$ by~(\xppy{n}).
On the other hand, if we begin by
applying~\eqref{d.1=:xy|->}, we get $(xp)\,q^n (tpy)_R.$
Since $tp=1-ps$ in $R,$ we have $(tpy)_R = y - (psy)_R,$
and this leads to a decomposition of $(xp)\,q^n (tpy)_R$
as the difference of two terms, one of which is $(xp)\,q^n y,$
while the other, $(xp)\,q^n(psy)_R,$ can be reduced as just mentioned.
So both reductions lead to $(xp)\,q^n y$ minus the result of reducing
$(xp)\,q^n(psy)_R$ using~(\xppy{n}), showing
that this ambiguity is also resolvable.

And the ambiguities coming from words $(xp)\,q^n\cdot(ps),$ which can
be reduced either by applying~(\xppy{n}) to the whole expression,
or~\eqref{d.1=:ps|->} to the final factor, are resolvable because of
the reductions~(\xpTP{n}), which were introduced precisely to
handle them.
(These ambiguities are, incidentally,
what are called in~\cite{<>} ``inclusion ambiguities'', where
the input-word of one reduction is a subword of the input-word of
another reduction.
All other ambiguities occurring in this note are
``overlap ambiguities''.)

We are left with the ambiguities resulting from the
overlap of two words both of which admit reductions in
our \eqref{d.1=:xpqpy|->}-series and/or
our \eqref{d.1=:xpqtp|->}-series.

Again, some of these are fairly straightforward to show resolvable.
Consider first a word $(xp)\,q^m\cdot y \cdot q^n (pz)$
where $y=py'=y''p\in B_{++},$ with $m,n\geq 2,$ to which we can apply
either~(\xppy{m}) on the left, or~(\xppy{n}) on the right.
It is not hard to verify that whichever of those operations
we apply first, the other will then be applicable to all
the words in the resulting expression.
(For instance, though the result of first applying~(\xppy{m})
will include some terms in which the factor $y=py'$ has been
absorbed into a product $(tpy')_R,$ the relation $py'=y''p$
allows us to rewrite this as $(ty''p)_R,$ so it lies in $Rp,$
hence is a $\!k\!$-linear combination of generators
in $B_{++}\cup B_{-+},$ allowing a subsequent application
of~(\xppy{n}) to each term.)
One finds that the result of either order of
reductions is the sum of a set of terms which can be constructed
as follows:
Starting with the word $(xp)\,q^m\cdot y \cdot q^n (pz),$
``mark'' an arbitrary subset of the $\!q\!$'s, subject
to the condition that the marked subset of the first $m$
$\!q\!$'s be nonempty and contain no pair of adjacent $\!q\!$'s,
and that the marked subset of the last $n$
$\!q\!$'s likewise be nonempty and contain no adjacent pair.
Now, as before, we make appropriate replacements involving the marked
$\!q\!$'s.
What most of these should be are clear from
the statement of Corollary~\ref{C.identity}.
For instance, if the first $q$ is marked,
replace $(xp)\,q$ by $(xps)_R;$
if the last $q$ is marked, replace $q\,(pz)$ by $(tpz)_R;$
if a $q$ that is neither initial nor final in the string
$q^m$ or $q^n$ is marked, replace it with $(tps)_R.$
Likewise, if the last $q$ before the factor $y$ is marked,
but the $q$ following the $y$ is not, then we replace
$q\,y=q\,py'$ by $(tpy')_R,$ while if
the $q$ following the $y$ is marked but not the one that
precedes it, we replace $y\,q=(y''p)\,q$ by $(y''ps)_R.$
But what if both of those $\!q\!$'s are marked?
Then we find that the results of performing either of those
two replacements, followed by the reductions
corresponding to the other, give the same result.
Indeed, the replacements correspond to two ways of
reducing $(tp)\,q\,y\,q\,(ps)\,,$ which is
an instance of~\eqref{d.xpqyqpz}, the resolvability of
which was verified in~\S\ref{S.norm};
both reductions of that factor give $(tys)_R.$
So this ambiguity is also resolvable.

The corresponding ambiguities with $m$ and/or $n$ equal to $1$
are resolved in the same way, with the obvious
adjustments; e.g., if $m=1,$ then instead of $q(py')$ being replaced
by $(tpy')_R$ in some terms of~\eqref{d.A,B},
we will have $(xp)\,q\,(py')$ replaced by $(xpy')_R.$
(The case $m=n=1$ is precisely~\eqref{d.xpqyqpz}.)

And the ambiguities $(xp)\,q^m\cdot y \cdot q^n (tp),$
which can be reduced by applying~(\xppy{m}) on
the left, or~(\xpTP{n}) on the right, are handled like the above,
{\em mutatis mutandis}.

There remain two sorts of ambiguities, which get a bit more interesting.
These will be given as~\eqref{d.mspn} and~\eqref{d.mspn+} below,
but let us prepare for them with the following observation.
Up to this point, ambiguities involving
reductions~(\xppy{n}) and/or~(\xpTP{n}) for certain
values of $n$ were resolved using only reductions indexed by
the same value(s) of $n$ and the value $1.$
Now if this were the case for the remaining sorts of ambiguities
as well, then for any set $N$ of positive integers containing~$1,$
the system of reductions given by~\eqref{d.1=:xy|->},
\eqref{d.1=:ps|->}, and the~(\xppy{n}) and~(\xpTP{n}) for $n\in N$
would have all ambiguities resolvable, and so would determine a
basis of monomials for the $\!k\!$-algebra presented by the generating
set $B\cup\{q\}$ and the relations corresponding to those reductions.
But we have seen that the relations corresponding to~\eqref{d.1=:xy|->},
\eqref{d.1=:ps|->}, and~(\xppy{1}) are sufficient to
present $R',$ in which all of the~(\xppy{n}) and~(\xpTP{n}) are
satisfied; so all such subsystems of our system of reductions
would yield $\!k\!$-bases for $R'.$
Yet some of these bases of $R'$ (those arising from larger sets $N)$
would be properly contained in others (arising from smaller sets $N),$
since a larger set of reductions would make more words reducible.

Since a proper inclusion among bases of
$R'$ is impossible, it {\em must} be true that in resolving some
of the ambiguities we have not yet considered,
reductions with larger subscripts than those
involved in the ambiguities themselves must be used.
And indeed, we saw in the preceding section that trying to
resolve the ambiguity $(xp)\,q\cdot(tp)\cdot q\,(py),$
arising from the reductions~(\xpTP{1}) and~(\xppy{1}), required
us to introduce the new reduction~(\xppy{2}).
So with the expectation that this will happen, let us look
at the two remaining families of ambiguities.

Consider first an ambiguously reducible word of the form
\begin{equation}\begin{minipage}[c]{35pc}\label{d.mspn}
$(xp)\,q^m\cdot(tp)\cdot q^n\,(py).$
\end{minipage}\end{equation}
In analyzing the effects of our reductions on this
expression, we shall use both of the notations introduced
in~\eqref{d.A,B} and~\eqref{d.C,D}.
Observe that in~\eqref{d.A,B}, the term $A((xp),n)(py)$
arises from those ways of marking $q^n$ such that
at least one $q$ is marked, but the last $q$ is {\em not} marked
(since $(py)$ has not been absorbed in a term $(tpy)_R),$
while $B((xp),n)(tpy)_R$ arises from those ways of marking $q^n$
under which a set of $\!q\!$'s {\em including} the last $q$ is marked.
Similarly, the two terms of~\eqref{d.C,D} arise from ways of
marking $q^n$ so that the {\em first}
$q$ is not, respectively, is, marked.
Note, finally that in the notation of~\eqref{d.A,B}, the output of
the reduction~(\xpTP{m}) is
\begin{equation}\begin{minipage}[c]{35pc}\label{d.spA,B}
$(xp)\,q^m - A((xp),m)(ps) - B((xp),m)(tps)_R.$
\end{minipage}\end{equation}

Now starting with the monomial~\eqref{d.mspn}, if we apply, on the one
hand,~(\xpTP{m}) with its output written as~\eqref{d.spA,B}
to the terms surrounding the first dot, and, on the other
hand,~(\xppy{n}) with output written as in~\eqref{d.C,D}
to the terms surrounding the second dot, then the equation equating the
results, which we hope to show can be established by further reductions
in our system, is
\begin{equation}\begin{minipage}[c]{35pc}\label{d.mspnABCD}
$(xp)\,q^{m+n}\,(py)\ -\ A((xp),m)\,(ps)\,q^n\,(py)\ -
\ B((xp),m)\,(tps)_R\,q^n\,(py)\\[.3em]
\hspace*{5em}=\ (xp)\,q^m (tp)\,C(n,(py))\ +
\ (xp)\,q^m (tps)_R\,D(n,(py)).$
\end{minipage}\end{equation}

The very first term above is the input of the
reduction~(\xppy{m+n}), so our ambiguity will be resolvable
if, on moving the other terms
of~\eqref{d.mspnABCD} to the right-hand side,
the value we end up with there, namely
\begin{equation}\begin{minipage}[c]{35pc}\label{d.mspn_1}
$A((xp),m)\,(ps)\,q^n\,(py)\ +\ B((xp),m)\,(tps)_R\,q^n\,(py)\ +\\[.3em]
\hspace*{5em}(xp)\,q^m (tp)\,C(n,(py))\ +\ (xp)\,q^m (tps)_R\,D(n,(py))$
\end{minipage}\end{equation}
can be reduced to the output of~(\xppy{m+n}).
We can, in fact, recognize two of the
terms of~\eqref{d.mspn_1} as parts of that output.
Namely, $B((xp),m)\,(tps)_R\,q^n\,(py)$ can be seen to be
the sum of all terms gotten from $(xp)\,q^{m+n}\,(py)$
by marking some subset of the first $m$ $\!q\!$'s
which includes the $\!m\!$-th, and then making the replacements
described in Corollary~\ref{C.identity}.
Likewise, $(xp)\,q^m (tps)_R\,D(n,(py))$ is the sum of the terms
we get on marking some subset of the last $n$ $\!q\!$'s
which includes the first of these, and making the
appropriate replacements.
So it suffices to show that the first and
third terms of~\eqref{d.mspn_1} can be reduced to the
sum of the other terms in the output of~(\xppy{m+n}).
(As they stand, the do not consist of such terms, since terms in
the outputs of our \eqref{d.1=:xpqpy|->}-series reductions
have no internal factors $(ps)$ or $(tp).)$

The term $A((xp),m)\,(ps)\,q^n\,(py)$ can be reduced by
applying~\eqref{d.1=:ps|->} to the factor $(ps),$ while the
term $(xp)\,q^m\,(tp)\,C(n,(py))$ can be reduced by
applying~(\xpTP{m}) to its initial factors.
Then these two remaining terms of~\eqref{d.mspn_1} become
\begin{equation}\begin{minipage}[c]{35pc}\label{d.mspn_2}
$A((xp),m)\,q^n\,(py)\ -\ A((xp),m)\,(tp)\,q^n\,(py)\ +\\[.3em]
\hspace*{2em}(xp)\,q^m\,C(n,(py))
-\ A((xp),m)\,(ps)\,C(n,(py))\ -\ B((xp),m)\,(tps)_R\,C(n,(py)).$
\end{minipage}\end{equation}

Of these terms, the first can be seen to consist of those summands
in the output of~(\xppy{m+n}) in which, again, only a subset
of the first $m$ $\!q\!$'s have been marked, this time a
subset which does {\em not} include the $\!m\!$-th $q;$ and the third
term likewise consists of those summands in which a subset
of the last $n$ $\!q\!$'s have been marked, but that
subset does not include the first of these.
Dropping these two terms from our calculation,
the first of the remaining terms can be reduced using~(\xppy{n}).
After performing that reduction, and again using~\eqref{d.C,D}
to describe the output, the terms that remain to be considered
take the form
\begin{equation}\begin{minipage}[c]{35pc}\label{d.mspn_3}
$- A((xp),m)\,(tp)\,C(n,(py))
- A((xp),m)\,(tps)_R\,D(n,(py))\\[.3em]
\hspace*{2em}- A((xp),m)\,(ps)\,C(n,(py))
- B((xp),m)\,(tps)_R\,C(n,(py)).$
\end{minipage}\end{equation}
The first and third of these can be combined
using the relation $ps+tp=1$ (or, formally, by applying
the reduction~\eqref{d.1=:ps|->} to the latter, and adding
the result to the former), giving
\begin{equation}\begin{minipage}[c]{35pc}\label{d.mspn_4}
$- A((xp),m)\,C(n,(py))
- A((xp),m)\,(tps)_R\,D(n,(py))
- B((xp),m)\,(tps)_R\,C(n,(py)).$
\end{minipage}\end{equation}

I claim that this expression gives precisely the remaining
terms of the output of~(\xppy{m+n}), i.e., those in which
there are marked $\!q\!$'s among both the first $m$ and the last $n.$
Indeed, the first summand in~\eqref{d.mspn_4} contains the terms of
this sort in which neither the $\!m\!$-th nor the $\!(m{+}1)\!$-st $q$
is marked; the next gives those in which the $\!m\!$-th $q$ is
again not marked, but the $\!(m{+}1)\!$-st is, and the last
gives those in which the $\!m\!$-th is marked, but
not the $\!(m{+}1)\!$-st.
Since adjacent $\!q\!$'s cannot both be marked, these are
all the possibilities.

But are the negative signs in~\eqref{d.mspn_4} what we want?
I agonized over this till I finally saw that they are.
In the description of the output of our reductions
in the \eqref{d.1=:xpqpy|->}-series, a term is assigned
a minus sign if it involves an {\em even} number of marked $\!q\!$'s,
a plus sign if it involves an {\em odd} number.
(This is the result of moving these terms of the expression proved
to sum to zero in Lemma~\ref{L.identity} to the opposite
side of the equation from $(xp)\,q^n (py).)$
Hence when we form a product such as
$A((xp),m)\,C(n,(py))$ in~\eqref{d.mspn_4},
the terms involving an even total number of marked $\!q\!$'s
will end up with a plus sign and those with an odd total number
will have a minus sign.
So this and the other terms of~\eqref{d.mspn_4}
indeed need negative signs to give the corresponding
summands in the output of~(\xppy{m+n}).

Once again, the verifications of the cases where $m$ and/or $n$ is
$1$ differ only in minor formal details from the above.

And the resolvability of the one remaining sort of ambiguity,
\begin{equation}\begin{minipage}[c]{35pc}\label{d.mspn+}
$(xp)\,q^m\cdot(tp)\cdot q^n\,(tp)$
\end{minipage}\end{equation}
reduces to the resolvability
of the case of~\eqref{d.mspn} where $(py)$ is $(ps),$
via Corollary~\ref{C.identity_tp}.
These computations establish

\begin{theorem}\label{T.1=}
Let $R$ be a $\!k\!$-algebra, let $p$ be an element of $R$
such that $1\in pR+Rp$ but $1\notin pR$ and $1\notin Rp,$
choose $s,t\in R$ such that $1=ps+tp$ as
in~\eqref{d.1=},
let $B\cup\{1\}$ be a spanning set for $R$
satisfying~\eqref{d.1=:B=} and~\eqref{d.1=:B_}, and let
\begin{equation}\begin{minipage}[c]{35pc}\label{d.1=:R'}
$R'\ =\ R\lang q\mid pqp=p\rang.$
\end{minipage}\end{equation}

Then $R'$ has a $\!k\!$-basis given by all
words in the generating set $B\cup\{q\}$ that contain no
subwords of any of the following forms:
\begin{equation}\begin{minipage}[c]{35pc}\label{d.1=:xy}
$xy$ \quad with $x,y\in B,$
\end{minipage}\end{equation}
\begin{equation}\begin{minipage}[c]{35pc}\label{d.1=:ps}
$(ps),$
\end{minipage}\end{equation}
\begin{equation}\begin{minipage}[c]{35pc}\label{d.1=:xpqpy}
$(xp)\,q^n\,(py)$ \quad with $xp\in B_{++}\cup B_{-+},$
$py\in B_{++}\cup B_{+-},$ and $n\geq 1,$
\end{minipage}\end{equation}
\begin{equation}\begin{minipage}[c]{35pc}\label{d.1=:xpqtp}
$(xp)\,q^n\,(tp)$ \quad with $xp\in B_{++}\cup B_{-+}$
and $n\geq 1.$
\end{minipage}\end{equation}

The reduction to the above normal
form may be accomplished by the systems of
reductions \eqref{d.1=:xy|->},
\eqref{d.1=:ps|->},
\textup{(\xppy{n})}~\textup{(}as shown in
Corollary~\ref{C.identity}, but with~\eqref{d.1=:xpqpy|->}
also included as \textup{(\xppy{1}))} and
\textup{(\xpTP{n})} \textup{(}as shown in
Corollary~\ref{C.identity_tp}, but with~\eqref{d.1=:xpqtp|->}
also included as \textup{(\xpTP{1}))}.\qed
\end{theorem}

Combining this with the results of \S\ref{S.norm} and~\S\ref{S.1-sided},
we see that we have determined the structure of
\mbox{$R\lang q\mid pqp=p\rang$} for all cases of a
$\!k\!$-algebra $R$ and an element $p\in R.$

\section{Some consequences, and a couple of loose ends}\label{S.&}

The proof of Theorem~\ref{T.1=} extends without difficulty to give
the analog of Proposition~\ref{P.M_norm}, with
``$\!p\!$-tempered $\!R\!$-module'' still defined as in
Definition~\ref{D.tempered}.
As in that proposition, we take a $\!k\!$-basis $C=C_+\cup C_-$
for $M,$ and supplement the reductions we
have used in the normal form for $R'$ with corresponding
reductions in which the leftmost factor, $x$ or $(xp),$
is replaced with an element of $C;$
an arbitrary element in the case of the $x$ of~\eqref{d.1=:xy},
a member of $C_+$
in the case of the $(xp)$ of~\eqref{d.1=:xpqpy} and~\eqref{d.1=:xpqtp}.

We will not write the result out in detail;
but let us note a common feature of this and our other results on the
extension of $\!p\!$-tempered $\!R\!$-modules to $\!R'\!$-modules.

\begin{corollary}\label{C.tempered}
If $R$ is a $\!k\!$-algebra, $p$ any element of $R,$ and
$(M,M_+)$ a $\!p\!$-tempered $\!R\!$-module, as defined
in Definition~\ref{D.tempered}, then the canonical
map $M\to (M,M_+)\otimes_{(R,p)} R'$ is an embedding such
that \textup{(}identifying $M$ with its image under that map\textup{)}
we have $M_+= M\cap(M\otimes_{(R,p)} R')\,p.$\qed
\end{corollary}

We can also ask when an inclusion of $\!k\!$-algebras leads to an
embedding of extensions of these algebras by universal inner inverses.
This is answered by

\begin{proposition}\label{P.R0,R1}
Let $R^{(0)}\subseteq R^{(1)}$ be an inclusion of $\!k\!$-algebras, and
$p$ an element of $R^{(0)}.$
Then the following conditions are equivalent.\\[.5em]
\textup{(a)}  The induced map $R^{(0)}\lang q\mid pqp=p\rang
\ \to\ R^{(1)}\lang q\mid pqp=p\rang$ is an embedding.\\[.5em]
\textup{(b)}
$R^{(0)}\cap(R^{(1)} p) = R^{(0)} p,$ and
$R^{(0)}\cap(pR^{(1)}) = pR^{(0)},$ and
$R^{(0)}\cap (p R^{(1)}{+}R^{(1)} p) = p R^{(0)}{+}R^{(0)} p.$
\end{proposition}

\begin{proof}
The direction (a)$\implies$(b) will not use our normal
form results, and, indeed, holds without the assumption
that $k$ is a field.
Observe that in any ring, if an element $p$ has an
inner inverse $q,$ then an element $x$ is right divisible
by $p$ if and only if $x(1-qp)=0:$
``only if'' is clear, while ``if'' holds because the
indicated equation makes $x=xqp.$
(We used the same idea in the proof of Corollary~\ref{C.M->MOX}.)
Under the assumption~(a), this immediately gives the first equality
of~(b); the second is seen dually.
The third holds by the similar criterion saying that
$x\in pR+Rp$ if and only if $(1-pq)\,x\,(1-qp)=0,$ where ``if''
holds because that equation can be written $x=pqx+xqp-pqxqp.$

The proof that (b)$\implies$(a) will use our normal form results.
Under the assumptions of~(b), note that whichever of the cases
``$\!1\notin pR+Rp$'', ``$\!1\in pR+Rp-(pR\cup Rp)$'',
``$\!1\in pR$'', ``$\!1\in Rp$''
apply to $R^{(1)}$ will also apply to $R^{(0)}.$

Let us now take a generating set $B^{(0)}=
B_{++}^{(0)}\cup B_{+-}^{(0)} \cup B_{-+}^{(0)}\cup B_{--}^{(0)}$
for $R^{(0)}$ as in our development of the case under which
$R^{(0)}$ and $R^{(1)}$ both fall.
(Some of these sets will be empty if we are in a case where $p$ is right
and/or left invertible.)
It follows from~(b) that we can extend each of
$B_{++}^{(0)},$ $B_{-+}^{(0)},$ $B_{+-}^{(0)},$ $B_{--}^{(0)}$
to a subset
$B_{++}^{(1)},$ $B_{-+}^{(1)},$ $B_{+-}^{(1)},$ $B_{--}^{(1)}$
of $R^{(1)}$ satisfying the corresponding conditions,
so as to yield a generating
set $B^{(1)}$ for $R^{(1)}$ with each component
containing the corresponding component of the generating
set $B^{(0)}$ for $R^{(0)}.$
Using these generating sets, the normal form expression
for each element of $R^{(0)}\lang q\mid pqp=p\rang$
is also the normal form of its image in
$R^{(1)}\lang q\mid pqp=p\rang.$
Hence, if an element is nonzero in the former
ring, so is its image in the latter, establishing~(a).
\end{proof}

(Incidentally, the first two equalities of~(b) above do not
imply the third.
For a counterexample, let $R^{(1)}$ be the Weyl algebra,
written as in~\eqref{d.Weyl} below, and $R^{(0)}$ its
subalgebra $k[p].$
Then $1\notin pR^{(0)}+R^{(0)}p$ but $1\in pR^{(1)}+ R^{(1)}p,$ so
the third condition of~(b) fails, though the first two clearly hold.
More on the algebra gotten by adjoining an inner inverse to $p$ in
the Weyl algebra in the next section.)

Is there a generalization of the above proposition based on a concept of
a ``$\!p\!$-tempered $\!k\!$-algebra'' $R,$
in which certain $\!k\!$-subspaces of $R$
are specified whose elements are to be treated like
right and/or left multiples of $p$?
A difficulty is that although, when we are dealing with genuine
left and right multiples of $p,$ reductions
$(xp)\,q\,(py)\mapsto (xpy)_R$ turn out to be well-defined, there is
no evident reduction of $xqy$ when $x$ and $y$ are
elements ``to be regarded as'' a right and a left multiple of $p$
respectively.
But I have not looked closely at the question.
\vspace{.3em}

Let's clear up a couple of loose ends.
I mentioned in the preceding section that the formulation of
Lemma~\ref{L.identity} used there was rigged for quick application.
Here is the promised more abstract version.
Note that the $n$ of the result below corresponds
to $n{-}1$ in Lemma~\ref{L.identity}, since there are $n{-}1$
places in which to insert factors between $n$ $q$'s.
The proof is essentially as before.

\begin{corollary}[to proof of Lemma~\ref{L.identity}]\label{C.multilin}
Let $n\geq 1,$ let $A$ and $A'$ be abelian groups,
let $\mu: A^n\to A'$ be an $\!n\!$-linear map,
and let $x$ and $y$ be elements of $A.$
Let $S(n)$ be the family of elements $\pm\,\mu(a_1,\dots,a_n)\in A'$
which arise from all ways of choosing each $a_i$ from the
$\!3\!$-element set $\{x,y,x{+}y\},$ and also choosing the
sign plus or minus, so as to satisfy the following conditions.
\begin{equation}\begin{minipage}[c]{35pc}\label{d.insert'}
If $(a_1,\dots,a_n)$ has an $x$ in a nonfinal position, it
has a $y$ in the next position.

If $(a_1,\dots,a_n)$ has a $y$ in a noninitial position, it
has an $x$ in the preceding position.

If the number of occurrences in $(a_1,\dots,a_n)$
of the substring ``$x,y\!$'', plus the
number of occurrences of initial $y$ and/or final $x,$
is odd, then the sign appended
to $\mu(a_1,\dots,a_n)$ is $-;$ otherwise it is $+.$
\end{minipage}\end{equation}

Then the sum of the resulting set $S(n)$ of elements
$\pm\,\mu(a_1,\dots,a_n)\in A'$ is $0.$

\textup{(}Above, if two of $x,\,y,\,x{+}y\in A$
happen to be equal, we treat them as formally distinct
in interpreting~\eqref{d.insert'}.\textup{)}\qed
\end{corollary}

Is the above the nicest version of the result?
A ``cleaner'' form would be the special
case where $A$ is the free abelian group on $\{x,y\},$ and
$A'$ the $\!n\!$-fold tensor power of $A,$ since
the form given above can be obtained from that case by composing
with maps into general $A$ and $A';$ and that case would
avoid distractions when studying the combinatorics of the result.
On the other hand, the form given above simplifies
applications such as we are making here.
\vspace{.5em}

Turning back to the proof of resolvability of the
ambiguities~\eqref{d.mspn} and~\eqref{d.mspn+}, it might
be possible to make this cleaner by first obtaining
identities involving the families
$S(n)\subseteq k\lang p,s,t,q\rang.$
If we let $S'(n)$ denote the subset of $S(n)$ consisting of those
elements, in the construction of which the final $q$ was {\em not}
marked (and, for convenience, define $S'(1)=\{q\}),$
then we can express the $S(n)$ in terms of these sets:
\begin{equation}\begin{minipage}[c]{35pc}\label{d.S=S'-}
$S(n)\ =\ S'(n)\ \cup\ -S'(n{-}1)\,(tp)\,q \quad (n\geq 2),$
\end{minipage}\end{equation}
and give a recursive construction of the $S'(n):$
\begin{equation}\begin{minipage}[c]{35pc}\label{d.S'12}
$S'(1)=\{q\},\qquad S'(2)=\{q\,(ps{+}tp)\,q,\ q\,(ps)\,q\},$
\end{minipage}\end{equation}
\begin{equation}\begin{minipage}[c]{35pc}\label{d.S'_rec}
$S'(n)\ =\ S'(n{-}1)\,(ps{+}tp)\,q\ \cup\ -S'(n{-}2)\,(tp)\,q\,(ps)\,q
\quad (n\geq 3).$
\end{minipage}\end{equation}

We would likewise let $S''(n)\subseteq S(n)$ be the subset
determined by the condition the that {\em initial} $q$ not be marked,
and give the corresponding formulas for these sets; and
we could probably develop formulas which,
mapped to our ring $R',$ would be equivalent to the resolvability
of our ambiguities.
If the method of the preceding section should prove useful beyond the
particular results we obtain there, this approach might be
worth pursuing.

\section{What if $R$ is the Weyl algebra?  Don't ask!}\label{S.Weyl}

A well-known example of a ring with an element $p$ that is
neither left nor right invertible, but which satisfies $1\in pR+Rp,$
is the Weyl algebra.
This is usually denoted $A=k\lang x,y\mid yx=xy+1\rang$
or $A=k\lang x,\,d/dx\rang;$ but for consistency with the
notation in the rest of this note, let us write it
\begin{equation}\begin{minipage}[c]{35pc}\label{d.Weyl}
$R\ =\ k\lang\,p,s\mid ps + (-s)p = 1\rang\,.$
\end{minipage}\end{equation}

It is natural to ask whether we can get a
nice normal form for the extension
\begin{equation}\begin{minipage}[c]{35pc}\label{d.W:R'=}
$R'\ =\ k\lang\,p,s,q\mid ps + (-s)p = 1,\ pqp=p\rang\,.$
\end{minipage}\end{equation}

If we want to apply the construction of Theorem~\ref{T.1=},
we first need to determine the $\!k\!$-subspaces
$pR$ and $Rp$ of $R,$ and their intersection.
It is a standard result that a $\!k\!$-basis of the Weyl
algebra is given by
\begin{equation}\begin{minipage}[c]{35pc}\label{d.W:s^mp^n}
$\{s^m p^n\mid m,n\geq 0\}.$
\end{minipage}\end{equation}
Indeed, every element of $R$ can be reduced to
a linear combination of members of this basis by
repeated application of the reduction
\begin{equation}\begin{minipage}[c]{35pc}\label{d.W:ps|->}
$ps\ \mapsto\ sp+1,$
\end{minipage}\end{equation}
which has no ambiguities.

Since right multiplying a $\!k\!$-linear combination of
elements of~\eqref{d.W:s^mp^n} by $p$ gives a $\!k\!$-linear
combination of such elements having $n>0,$
it follows that $Rp$ is precisely the $\!k\!$-subspace of $R$
spanned by the elements $s^m\,p^n$ with $n>0.$

One can characterize $pR$ similarly using the basis
$\{p^n s^m\mid m,n\geq 0\},$ but this does not help if we want to
study both subspaces at the same time.
So let us,
for now, represent elements of $R$ using the basis~\eqref{d.W:s^mp^n},
and investigate what linear combinations of these
basis elements lie in $pR.$

If we apply~\eqref{d.W:ps|->} repeatedly starting with $p\,s^m,$ we get
\begin{equation}\begin{minipage}[c]{35pc}\label{d.W:ps^m}
$p\,s^m\ =\ s^m\,p\ +\ m\,s^{m-1}.$
\end{minipage}\end{equation}
Thus, $s^m p\equiv -\,m s^{m-1}\pmod{pR},$ and right multiplying
this congruence by $p^{n-1},$ we get
\begin{equation}\begin{minipage}[c]{35pc}\label{d.W:s^mp^n==}
$s^m p^n\ \equiv\ -m\,s^{m-1}p^{n-1}\hspace*{-.6em}\pmod{pR}$\quad
for $m,n\geq 1.$
\end{minipage}\end{equation}

We can iterate~\eqref{d.W:s^mp^n==}, decreasing the exponents
of $s$ and $p$ until one of them goes to zero.
So if $n>m,$ we conclude that $s^m p^n$ is congruent
modulo $pR$ to an integer multiple of a positive power of $p;$
hence it lies in $pR;$ and since it also lies in $Rp,$ we get
\begin{equation}\begin{minipage}[c]{35pc}\label{d.W:n>m}
$s^m p^n\in pR\cap Rp$\quad if $n>m.$
\end{minipage}\end{equation}

On the other hand, if $m\geq n > 1,$ it is convenient to
iterate~\eqref{d.W:s^mp^n==} only to the point of bringing
the exponent of $p$ down to $1.$
That gives us
$s^m p^n\equiv (-1)^{n-1}\,m(m-1)\dots(m-n+2)\,s^{m-n+1}\,p\pmod{pR},$
so again, since both expressions lie in $Rp,$ we get
\begin{equation}\begin{minipage}[c]{35pc}\label{d.W:m_geq_n}
$s^m p^n\ \equiv\ (-1)^{n-1}\,m(m{-}1)\dots(m{-}n{+}2)\,s^{m-n+1}\,p
\pmod{pR\cap Rp}$\quad if $m\geq n > 1.$
\end{minipage}\end{equation}

Combining~\eqref{d.W:n>m},~\eqref{d.W:m_geq_n}, and the
vacuous relation $s^m\ \equiv\ s^m \pmod{pR\cap Rp},$ we get
\begin{equation}\begin{minipage}[c]{35pc}\label{d.W:pR+Rp_big}
Every element of $R$ is congruent modulo $pR\cap Rp$ to
a linear combination of words $s^m$ and~$s^m p.$
\end{minipage}\end{equation}

Since the family of words $\{s^m,\,s^m p\mid m\geq 0\}$ is
``small'' compared with the full $\!k\!$-basis~\eqref{d.W:s^mp^n},
we see that when we form our desired spanning set $B,$ ``most of''
that set can be expected to lie in the component $B_{++}\,.$

Further details depend on the characteristic of $k.$
We shall consider the case where $\r{char}(k)=0.$

In this case, solving~\eqref{d.W:ps^m} for $s^{m-1},$ we see
that every power of $s$ lies in $pR+Rp.$
Hence,
\begin{equation}\begin{minipage}[c]{35pc}\label{d.W:R=pR+Rp}
If $\r{char}(k)=0,$ then $R=pR+Rp.$
\end{minipage}\end{equation}
It is now easy to verify that
\begin{equation}\begin{minipage}[c]{35pc}\label{d.W:B=}
\hspace*{0em}If $\r{char}(k)=0,$ then a spanning set $B$ for $R$ with
the properties of~\eqref{d.1=:B=}, \eqref{d.1=:B_} is given
by\vspace{.2em}

$B_{++}\ =\ \{s^m p^n\mid n>m\geq 0\}\ \cup
\ \{s^m p^n + ms^{m-1}\,p^{n-1}\mid m\geq n>1\}$\quad
(cf.~\eqref{d.W:n>m} and~\eqref{d.W:s^mp^n==}),\vspace{.2em}

$B_{-+}\ =\ \{-s^m\,p\mid m>0\},$\vspace{.2em}

$B_{+-}\ =\ \{p\,s^m\mid m>0\}
\ =\ \{s^m\,p + m\,s^{m-1}\mid m>0\}$\quad
(cf.~\eqref{d.W:ps^m}),\vspace{.2em}

$B_{--}\ =\ \emptyset.$
\end{minipage}\end{equation}
(I have put a minus sign into the entries of $B_{-+}$ to conform
with the convention made
in~\eqref{d.1=:B=}, that $B_{-+}$ contain $tp,$
which, in writing~\eqref{d.Weyl}, we have taken to be $(-s)p.)$

Using the above basis, we can obtain by Theorem~\ref{T.1=} a normal
form for $R'=k\lang p,s,q\mid ps+(-s)p= 1,\ \linebreak[1] pqp=p\rang.$

But what we would really like is a normal
form in terms of the generators $p,$ $s$ and $q.$
When first exploring
the case $1\in pR+Rp -(pR\cup Rp)$ of the subject of this note,
I took the Weyl algebra as a sample case, and
tried to find such a normal form;
but the ambiguities among reductions I obtained kept spawning
new reductions, without apparent pattern.
This, along with calculations showing
that the forms of these reductions must
depend on the characteristic of $k,$
led me to doubt for a long time that any reasonable
normal form could be found when $1\in pR+Rp -(pR\cup Rp).$
It was only when I dropped the case of the Weyl algebra,
and returned to consideration of a general $\!k\!$-algebra,
that I was able to get anywhere.

However, with the results of \S\ref{S.1=:norm} now at
hand, we can develop a normal form for this algebra $R'$
in terms of $p,$ $s$ and $q,$ and shall do so below
(still assuming $\r{char}(k)=0).$

(Let me here moderate the semi-facetious title of this section,
to merely say that if, at some point the reader chooses not to
slog further through the lengthy argument for the sake of a normal
form whose value is not evident, I will not argue with his
or her choice.)

In preparation for the result, let us note that the normal form
that we would get by simply applying Theorem~\ref{T.1=} to
the basis~\eqref{d.W:B=} for $R$ is somewhat
atypical among applications of that theorem.
In the general situation of Theorem~\ref{T.1=},
if we take from our spanning set $B$ three elements
$xp\in B_{++}\cup B_{-+},$ $y\in B_{--},$ $pz\in B_{++}\cup B_{+-}$
(note the choice of $y$ here, the opposite of what we considered
when looking for ambiguities!), then products $(xp)\,q^m\,y\,q^n\,(pz)$
are irreducible: the presence of $y\in B_{--}$ between
$(xp)$ and $(pz)$ blocks any reductions.
However, with a basis like~\eqref{d.W:B=}, where $B_{--}$ is empty,
no such blockage is possible; and we find that in any string
of elements of $B\cup\{q\}$ that is reduced with
respect to the normal form of Theorem~\ref{T.1=}, no element
of $B_{++}\cup B_{-+}$ can occur anywhere to the left of an element
of $B_{++}\cup B_{+-}.$
So the elements of $B$ (if any) occurring interspersed among
the $\!q\!$'s in our word
will begin with a sequence (possibly empty) of
members of $B_{+-},$ and end with a sequence (possibly empty) of
members of $B_{-+},$ with at most a single member of $B_{++}$
between these.

This will prepare us for the fact that words in the normal form
based on $p,$ $s$ and $q$ that we shall obtain will
typically have a sort of singularity in the middle.
To prepare us in a more detailed way for the form they will have,
let us note that where in \S\ref{S.1=:norm}, a key tool was
to apply, between various pairs of $\!q\!$'s in a string $q\dots q,$
the relation $1=ps+tp,$ in the present situation we can, more
generally, whenever
an $s^m$ $(m\geq 0)$ appears between $\!q\!$'s, apply the result of
putting $m{+}1$ for $m$ in~\eqref{d.W:ps^m} and solving for~$s^m:$
\begin{equation}\begin{minipage}[c]{35pc}\label{d.W:s^m=}
$s^m\ =\ (m+1)^{-1}(ps^{m+1} - s^{m+1}p).$
\end{minipage}\end{equation}

Using these ideas, we shall now prove

\begin{theorem}\label{T.W}
Let $k$ be a field of characteristic~$0.$
Then the algebra
\begin{equation}\begin{minipage}[c]{35pc}\label{d.W:R'=.}
$R'\ =\ k\lang p,s,q\mid ps=sp+1,\ pqp=p\rang$
\end{minipage}\end{equation}
has a $\!k\!$-basis consisting of
all words in $p,$ $s$ and $q$ in which no $p$
is immediately followed by an $s,$ and the $\!p\!$'s that
occur \textup{(}if any\textup{)} form a single consecutive string.
In other words, every such word has the form
\begin{equation}\begin{minipage}[c]{35pc}\label{d.W:n_form}
$s^{a_0}\,q\,s^{a_1}\,q\,\dots\,q\,s^{a_{m-1}}\,q\,s^{a_m}\,p^b\,q\,
s^{a_{m+1}}\,q\,\dots\,q\,s^{a_n},$
\end{minipage}\end{equation}
where $0\leq m\leq n,$ $b\geq 0,$ and all $a_i\geq 0.$
\textup{(}Remark: If $b=0,$ then $m$ is,
of course, not uniquely defined.\textup{)}
\end{theorem}

\begin{proof}
We shall first show that every monomial in our generators can be
reduced to a linear combination of
monomials~\eqref{d.W:n_form}, so that these
span $R',$ then that the set of such
monomials is $\!k\!$-linearly independent.
We will not follow the formalism of the Diamond Lemma, though
some of the ideas will be similar.
In particular, in the first part of our proof,
we shall associate to every monomial a $\!4\!$-tuple of
natural numbers, and show that every monomial not of the
form~\eqref{d.W:n_form} is equal in $R'$ to a $\!k\!$-linear
combination of monomials each of which has smaller
associated $\!4\!$-tuple, under lexicographic ordering.
This is enough to show that every monomial is a $\!k\!$-linear
combination of monomials~\eqref{d.W:n_form}.
(If not every monomial were so expressible,
there would be a least $\!4\!$-tuple associated
with a counterexample monomial $w,$ and applying a reduction of
the indicated sort to $w$ would give a contradiction.)

The $\!4\!$-tuple we shall associate with a word $w$ is
\begin{equation}\begin{minipage}[c]{35pc}\label{d.aaab}
$h(w)\ =\ (a_q(w),\,a_p(w),\,a_s(w),\,b_{p,s}(w)),$
\end{minipage}\end{equation}
where the first three coordinates are the numbers of
$\!q\!$'s, $\!p\!$'s and $\!s\!$'s in $w,$ and the last is
the number of occurrences of a $p$ anywhere before an $s,$
i.e., the number of ordered pairs $(i,j)$ with $i<j$ such
that the $\!i\!$-th factor of $w$ is a $p$ and the $\!j\!$-th is an $s.$
This refinement of the coordinate ``number of occurrences
of the element $(ps)$'' that we used in \S\ref{S.1=:norm} is needed
here: if we simply counted occurrences of the
string $ps,$ calling this number $a_{ps}(w),$ then inequalities
involving this function would not respect formal multiplication
of monomials:  clearly, $a_{ps}(sp)<a_{ps}(ps),$ yet multiplying
these monomials
on the left by $p$ we find that $a_{ps}(psp)\not<a_{ps}(pps).$
However, I claim that for $h$ defined by~\eqref{d.aaab}, if
$h(u)\leq h(u')$ and $h(v)\leq h(v'),$ with
at least one of these inequalities strict, then $h(uv)<h(u'v').$
Indeed, this is obvious except in the case where the first three
coordinates of $h(u)$ agree with those of $h(u')$
and the first three coordinates of $h(v)$ agree with those of $h(v'),$
so that the comparison depends on the $\!4\!$-th coordinate, $b_{p,s}.$
Now it is easy to see that in general,
$b_{p,s}(uv) = b_{p,s}(u) + b_{p,s}(v) + a_p(u)\,a_s(v),$ so when
the $\!a\!$-coordinates
are the same for $u$ and $u',$ and likewise for $v$ and $v',$
the $\!b_{p,s}\!$-coordinate of our product depends additively
on the $\!b_{p,s}\!$-coordinates of the factors, from which
the desired inequality follows.

So let us assume $w$ is a monomial not of the form~\eqref{d.W:n_form},
and prove that it is a linear combination of monomials
with smaller values of $h.$

If $w$ contains a sequence $ps,$ then applying the relation
$ps=sp+1,$ we get a sum of two monomials on each of which
$h$ clearly has value $<h(w).$

If $w$ has no subsequence $ps,$ then to fail
to have the form~\eqref{d.W:n_form}, it must have
two $\!p\!$'s with a nonempty string of non-$\!p\!$ terms between them.
Writing $u$ and $v$ for the (possibly empty) segments before and after
these two $\!p\!$'s, we can write
\begin{equation}\begin{minipage}[c]{35pc}\label{d.W:up...pv}
$w\ =\ u\ p\,q\,s^{m_1}\,q\,\dots\,q\,s^{m_{n-1}}\,q\,s^{m_n}\,p\ v,$
\quad where $n\geq 1$ and $s_{m_1},\dots,s_{m_n}\geq 0.$
\end{minipage}\end{equation}
It will now suffice to show that the string between $u$ and $v,$
\begin{equation}\begin{minipage}[c]{35pc}\label{d.W:p...p}
$p\,q\,s^{m_1}\,q\,\dots\,q\,s^{m_n}\,q\,s^{m_n}\,p$
\end{minipage}\end{equation}
is equal in $R'$ to a $\!k\!$-linear combination of words on which $h$
has lower values.

As a first step, let us use the relation~\eqref{d.W:ps^m} in
reverse, to replace the final $s^{m_n}\,p$ of~\eqref{d.W:p...p}
with $p\,s^{m_n} - m_n s^{m_n-1}$ if $m_n>0,$
turning~\eqref{d.W:p...p}
into a linear combination of two monomials.
One of these, the one arising from the $s^{m_n-1}$ term,
has a strictly lower value of $h,$ so we can ignore it.
The other,
\begin{equation}\begin{minipage}[c]{35pc}\label{d.W:p...ps^m}
$p\,q\,s^{m_1}\,q\,\dots\,q\,s^{m_{n-1}}\,q\,p\,s^{m_n}$
\end{minipage}\end{equation}
has value of $h$ that is higher than~\eqref{d.W:p...p},
but only in its $b_{p,s}$ coordinate.
I claim now that we can further rewrite~\eqref{d.W:p...ps^m}
so as to turn it into
a $\!k\!$-linear combination of monomials all involving
fewer $\!q\!$'s, and hence all having lower values of $h$
than~\eqref{d.W:p...p} has.
If $m_n=0$ (the case we temporarily excluded at the
start of this paragraph),~\eqref{d.W:p...ps^m}
is the same as~\eqref{d.W:p...p}; so in either case we have the latter
expression to consider.

If $n=1,$ then~\eqref{d.W:p...ps^m} has the form $p\,q\,p\,s^{m_1},$
which clearly equals $p\,s^{m_1},$ giving a decreased
value of $a_q,$ as desired.
For $n>1,$ the idea will be,
as indicated, to apply~\eqref{d.W:s^m=} to each
of the factors $s^{m_i}$ nestled between the $\!q\!$'s,
and then treat these as we treated the factors $1=ps+tp$
in~\S\ref{S.1=:norm}.
Now replacing $s^m_i$ by $(m_i+1)^{-1}(ps^{m_i+1} - s^{m_i+1}p)$
gives terms with larger values of $a_p,$ $a_s$ and (usually) $b_{p,s}$
than~\eqref{d.W:p...ps^m} had;
but it does not affect the value of $a_q,$
so we will still be safe if the result can be reduced
to a linear combination of terms all having lower values of $a_q.$

To formalize this process, we shall apply
Corollary~\ref{C.multilin}, with $n{-}1$ for the $n$ of
that corollary, taking for $A$ a
$\!2\!$-dimensional $\!k\!$-vector-space with basis written
$\{x,y\},$ and for $A'$ the underlying $\!k\!$-vector-space of $R'.$
Let us define $\!k\!$-linear maps $\mu_1,\dots,\mu_{n-1}: A\to R'$
by letting $\mu_i$ carry $x$ to $(m_i+1)^{-1}s^{m_i+1}\,p,$
and $y$ to $(m_i+1)^{-1}p\,s^{m_i+1},$ so that it
carries $x{+}y$ to $s^{m_i}.$
Define $\mu:A^{n-1}\to A'$ by
\begin{equation}\begin{minipage}[c]{35pc}\label{d.W:*m}
$\mu(a_1,\dots,a_{n-1})\ =
\ p\,q\,\mu_1(a_1)\,q\,\dots\,q\,\mu_{n-1}(a_{n-1})\,q\,p\,s^{m_n}.$
\end{minipage}\end{equation}
By Corollary~\ref{C.multilin}, the sum of the set $S(n{-}1)$
defined in that corollary using the map~\eqref{d.W:*m} equals zero.
We find that the term of that sum in which
all of $a_1,\dots,a_{n-1}$ are $x{+}y$ is exactly~\eqref{d.W:p...ps^m},
while in every other term, at least one of the $n$ $\!q\!$'s
has a $p$ before it and a $p$
after it, so that an application of the relation $pqp=p$ allows
us to reduce the number of $\!q\!$'s.
So~\eqref{d.W:p...ps^m} is equal to a linear combination of
monomials each involving fewer $\!q\!$'s, as claimed.
This completes our proof that the
elements~\eqref{d.W:n_form} span $R'.$

How shall we now show these elements $\!k\!$-linearly independent?
One approach would be to formalize the above argument
as giving a reduction system in the sense of the Diamond Lemma,
and verify that all its ambiguities are reducible.
But that verification was already tedious in the simpler context
of Theorem~\ref{T.1=}.

Rather, let us apply Theorem~\ref{T.1=} to
the generating set~\eqref{d.W:B=} of $R',$ and then
show that when the monomials~\eqref{d.W:n_form} are
expressed in terms of the basis given by
that theorem, they have distinct leading
terms, proving them $\!k\!$-linearly independent.

Of course, to define ``leading term'', we need a total ordering
on the basis of $R'$ in question.
To describe the ordering we will use,
let the ``weight'' of a member of
the basis $B$ of~\eqref{d.W:B=} be the highest exponent of $s$
appearing in its expression.
(E.g., the weight of $s^m p + m s^{m-1}\in B_{+-}$ is $m.)$
We now define a word $w$
in the elements of $B\cup\{q\}$ to be larger than a word $w'$ if
it involves more $\!q\!$'s; or if it involves the same number
of $\!q\!$'s but the total weight of the factors from $B$ is
higher, or if we have equality of both of these, but it has
more terms from $B_{+-};$ while when all of these are equal, let
the total ordering be chosen in an arbitrary fashion.

We now consider a word $w$ of the form~\eqref{d.W:n_form}, and
the operation of expressing it in the normal form
of Theorem~\ref{T.1=} determined by the basis~\eqref{d.W:B=} of
our Weyl algebra; and ask what its leading term with respect to
the above ordering will be.

First, suppose that $b,$ the exponent of $p$ in~\eqref{d.W:n_form},
is zero.
Then to write $w$ as an expression (not reduced, to start
with) in the elements of in $B\cup\{q\},$
we may replace every term $s^{a_i}$ with $a_i>0$ by
$(a_i+1)^{-1}((ps^{a_i+1}) - (s^{a_i+1}p)),$ while writing any
factors $s^{a_i}$ with $a_i=0$ as $1,$ the empty word.
When we multiply this expression out, every pair of successive
$\!q\!$'s are either adjacent, or have between them a generator
$(ps^{a_i+1})\in B_{+-}$ or $(s^{a_i+1}p)\in B_{-+}.$
Those of the resulting words that have a member of $B_{-+}$
anywhere to the left of a member of $B_{+-}$ can be reduced by
one of the reductions in our \eqref{d.1=:xpqpy|->}-series
to a linear combination of words involving smaller numbers of $\!q\!$'s.
Of those that remain, we see that the one that will be largest
under our ordering will be (by the stipulation
regarding elements of $B_{+-}$ in our description
of that ordering), the one with the greatest
number of factors from $B_{+-};$ i.e., the one in which
$(ps^{a_i+1})\in B_{+-}$ has been used in each position
where $a_i>0.$
Clearly this leading reduced word determines the sequence
of exponents $a_i,$ hence it uniquely determines $w.$

Next, suppose $b=1.$
The first step in expressing $w$ in terms of the
generators~\eqref{d.W:B=} is the same as before, except that
the factor $s^{a_m} p,$ unlike the factors $s^{a_i},$
is not modified, since it is, as it stands, a member of $B_{-+}.$
In this case, all the words we get that have a member of $B_{+-}$
{\em after} that term again have a member of $B_{-+}$
to the left of a member of $B_{+-},$
and so can be reduced to terms with fewer $\!q\!$'s,
so the terms that cannot be so reduced must have
factors $(s^{a_i+1}p)\in B_{-+}$ in those positions.
On the other hand, of the terms before $s^{a_m}\,p,$
the largest one under our ordering will again have all
factors from $B$ of the form $(ps^{a_i+1})\in B_{+-}.$
So the largest term occurring determines both the sequence of
$a_i$ and the position where the $p$ occurs in $w$ (namely,
the position where the first element of $B_{-+}$ appears).
Moreover, that leading term
is not equal to the leading term of an expression with $b=0,$
since as we have seen, the latter have no factors in $B_{-+}.$

Finally, if $b>1,$ we have behavior similar to the case $b=1,$ except
that the factor $s^{a_m}\,p^b$ now reduces to the sum
of an element of $B_{++}$ and possibly an expression lower under
our ordering.
(By the description of $B_{++}$ in~\eqref{d.W:B=},
such a lower summand will appear if $b\leq a_m.)$
Only the former summand need be looked at; and we see again
that the unique term having members of $B_{+-}$ before that
element of $B_{++},$ and members of $B_{-+}$ after it,
will be irreducible
under the normal form of Theorem~\ref{T.1=}, and will give
the leading term of our reduced expression.
This leading term now determines both the value of $b$ and, as before,
the values of $m$ and of the $a_i,$ and so again determines $w.$

This completes the proof of the Theorem.
\end{proof}

When $\r{char}(k)=e>0,$ things are somewhat different.
On the one hand,~\eqref{d.W:m_geq_n} simplifies pleasantly whenever
$e\,|\,m(m{-}1)\dots(m{-}n{+}2).$
On the other hand, I claim that the elements $s^m$ with
$m\equiv -1\pmod{e}$ are $\!k\!$-linearly independent modulo $pR+Rp.$
Indeed, since $R$ is spanned over $k$ by elements
$s^m p^n,$ the space $pR$ is spanned by elements
$p\,s^m p^n,$ and using~\eqref{d.W:ps^m} we see that
in the expansions of these elements in terms of
the basis~\eqref{d.W:s^mp^n},
basis elements $s^m$ with $m\equiv -1\pmod{e}$
never appear with nonzero coefficients.
Since they also certainly do not appear with nonzero coefficients in
the expressions in that basis for elements of $Rp,$
they do not appear in the expressions for elements of $pR+Rp.$
One finds that $\{s^m\mid m\equiv -1\pmod{e}\}$
can be taken as a basis of $B_{--}.$
Probably one can get a normal form for $R'$ somewhat like the
above; but with multiple clusters of $\!p\!$'s allowed, separated
by strings $q\,s^m\,q$ with $m\equiv -1\pmod{e}.$
However, I have not looked into this.

\section{Late addendum: mutual inner inverses}\label{S.mutual}

At about the time this paper was accepted for publication,
I received a preprint of~\cite{A+KOM}, in which P.\,Ara and K.\,O'Meara
used results in the preprint version of this note to answer
an open question on nilpotent regular elements in rings.
Their method required them to extend the result
of Theorem~\ref{T.1_notin}, for a certain $R,$
to get a description of the $\!k\!$-algebra generated over
that $R$ by a universal {\em mutual} inner inverse of $p,$
$R''=R\lang q\mid pqp=p,\,qpq=q\rang.$
This led me to wonder whether I could save them that
awkwardness, and get some useful general results,
by extending some of the
material of this paper to mutual inner inverses.
(Incidentally, what I am calling ``mutual inner inverses''
are more often called ``generalized inverses'', and
are so called in \cite{A+KOM}.
But I prefer
to use here a term that highlights their relation with inner inverses.)

The symmetry of the property of being mutually inner inverse
suggests that, just as $p$ is taken in Theorem~\ref{T.1_notin}
to be an element of a fairly general $\!k\!$-algebra $R,$ so $q$ might
be taken from another such $\!k\!$-algebra $S.$
And, indeed, it turns out that if such $p$ and $q$ are {\em nonzero}
and satisfy $1\notin pR+Rp,$ $1\notin qS+Sq,$
then we can build on Theorem~\ref{T.1_notin} to get a very
similar normal form for this construction.
In this normal form, we will, on the one hand, use
a $\!k\!$-basis $B$ for $R$ as in Theorem~\ref{T.1_notin}
(but note that in the present situation, the qualifying phrase
``if $p\neq 0$''
can be removed from the condition that $B_{++}$ contain $p,$
in the first line of~\eqref{d.B_}, since, as noted above, $p$
is here assumed nonzero).
Likewise, we will use a $\!k\!$-basis for $S$ of the analogous form,
\begin{equation}\begin{minipage}[c]{35pc}\label{d.C=}
$C\cup\{1\}\ =\ C_{++}\cup C_{+-} \cup C_{-+}\cup C_{--}\cup \{1\},$
\end{minipage}\end{equation}
where
\begin{equation}\begin{minipage}[c]{35pc}\label{d.C_}
$C_{++}$ is any $\!k\!$-basis of $qS\cap Sq$ which contains $q,$

$C_{+-}$ is any $\!k\!$-basis of $qS$ relative to $qS\cap Sq,$

$C_{-+}$ is any $\!k\!$-basis of $Sq$ relative to $qS\cap Sq,$

$C_{--}$ is any $\!k\!$-basis of $S$ relative to $qS+Sq+k.$
\end{minipage}\end{equation}

We can now state and prove

\begin{theorem}\label{T.mutual}
Suppose $R$ and $S$ are $\!k\!$-algebras
\textup{(}which for notational simplicity we will assume are disjoint
except for the common subfield $k),$ and let $p\in R-\{0\},$
$q\in S-\{0\}$ satisfy
\begin{equation}\begin{minipage}[c]{35pc}\label{d.1_notinx2}
$1\notin pR+Rp,\qquad 1\notin qS+Sq.$
\end{minipage}\end{equation}
Let $B\cup\{1\}$
be a $\!k\!$-basis for $B$ as in~\eqref{d.B=} and~\eqref{d.B_},
and $C\cup\{1\}$ a $\!k\!$-basis for $S$ as
in~\eqref{d.C=} and~\eqref{d.C_}.

Then the $\!k\!$-algebra $T$ freely generated by the two
$\!k\!$-algebras $R$ and $S,$ subject to the two additional relations
\begin{equation}\begin{minipage}[c]{35pc}\label{d.pqp,qpq}
$pqp\ =\ p,\quad qpq\ =\ q,$
\end{minipage}\end{equation}
has a $\!k\!$-basis given by all words in $B\cup C$ which
contain no subwords as in~\eqref{d.xy} or~\eqref{d.xpqpy}
\textup{(}that is, no subwords of the form $xy$ with $x,y\in B,$
or $(xp)\,q\,(py)$ with $xp\in B_{++}\cup B_{-+},$
$py\in B_{++}\cup B_{+-}),$ nor any subwords of the analogous forms
\begin{equation}\begin{minipage}[c]{35pc}\label{d.C:xy}
$xy$ \quad with $x,y\in C,$
\end{minipage}\end{equation}
or
\begin{equation}\begin{minipage}[c]{35pc}\label{d.C:xqpqy}
$(xq)\,p\,(qy)$ \quad with $xq\in C_{++}\cup C_{-+}$
and $qy\in C_{++}\cup C_{+-}\,.$
\end{minipage}\end{equation}

The reduction to the above normal form may be accomplished by
the reductions~\eqref{d.xy|->} and~\eqref{d.xpqpy|->}
of Theorem~\ref{T.1_notin}, together with the analogous reductions,
\begin{equation}\begin{minipage}[c]{35pc}\label{d.C:xy|->}
$xy\ \mapsto\ (xy)_S$ \quad for all $x,y\in C,$
\end{minipage}\end{equation}
and
\begin{equation}\begin{minipage}[c]{35pc}\label{d.C:xqpqy|->}
$(xq)\,p\,(qy)\ \mapsto\ (xqy)_S$ \quad for all
$xq\in C_{++}\cup C_{-+},$ $qy\in C_{++}\cup C_{+-}\,.$
\end{minipage}\end{equation}
\end{theorem}

\begin{proof}
It is clear that the reductions~\eqref{d.xy|->}, \eqref{d.xpqpy|->},
\eqref{d.C:xy|->} and~\eqref{d.C:xqpqy|->} correspond to relations
holding in $T,$ and include enough relations to present
that algebra, and that they all reduce the lengths of their input-words.
So it suffices to check that all ambiguities of
the resulting reduction system are resolvable.

Note that the input-word
of each of the reductions~\eqref{d.xy|->},~\eqref{d.xpqpy|->}
begins and ends with generators from $B,$ while the
input-words of~\eqref{d.C:xy|->} and~\eqref{d.C:xqpqy|->}
begin and end with generators from $C.$
Hence, if an ambiguity in our reduction system involves an
overlap of only one letter, the two words must either
both come from~\eqref{d.xy|->} and/or~\eqref{d.xpqpy|->}, or
both come from~\eqref{d.C:xy|->} and/or~\eqref{d.C:xqpqy|->}.
In the former case, that ambiguity will be resolvable by
Theorem~\ref{T.1_notin}, and in the latter case,
by that same theorem applied with $S,$ $q$ and $p$
in the roles of $R,$ $p$ and $q.$

It remains to consider two-letter overlaps.
We implicitly noted in the proof of Theorem~\ref{T.1_notin}
that there were no such overlaps involving only
reductions~\eqref{d.xy|->} and/or~\eqref{d.xpqpy|->};
so there are likewise none involving
only~\eqref{d.C:xy|->} and/or~\eqref{d.C:xqpqy|->}.
Hence two-letter overlaps must involve one reduction from the former
family and one from the latter.
However, the only generators appearing in both families
of reductions are $p$ and $q.$
From this it is easy to check that the remaining ambiguously
reducible monomials are precisely
\begin{equation}\begin{minipage}[c]{35pc}\label{d.xpqpqy}
$(xp)\cdot q\,p\cdot (qy),$\quad where\quad
$xp\in B_{++}\cup B_{-+},$
$qy\in C_{++}\cup C_{+-},$
\end{minipage}\end{equation}
and
\begin{equation}\begin{minipage}[c]{35pc}\label{d.xqpqpy}
$(xq)\cdot p\,q\cdot (py),$\quad where\quad
$xq\in C_{++}\cup C_{-+},$
$py\in B_{++}\cup B_{+-}.$
\end{minipage}\end{equation}

I claim that
the two competing reductions applicable to~\eqref{d.xpqpqy}
each reduce it to $(xp)(qy).$
Indeed, to reduce the initial string $(xp)\,q\,p$ in~\eqref{d.xpqpqy},
we write the factor $p$ as $(p1)\in B_{++}$ and
apply~\eqref{d.xpqpy|->}, getting $(xp)\,q\,(p1)\mapsto (xp1)_R=(xp);$
which reduces the product~\eqref{d.xpqpqy} to $(xp)(qy).$
The other reduction similarly applies~\eqref{d.C:xqpqy|->} to the final
string $q\,p\,(qy),$ and gives the same result.

Likewise, the two reductions applicable
to~\eqref{d.xqpqpy} both reduce it to $(xq)(py).$

Hence all the ambiguities of our reduction
system are resolvable, so $T$ has a
normal form given by the words irreducible under
that system; that is, those having no
subwords~\eqref{d.xy}, \eqref{d.xpqpy}, \eqref{d.C:xy} or
\eqref{d.C:xqpqy}, as required.
\end{proof}

The construction needed for \cite{A+KOM} can now be gotten as a
special case.

\begin{corollary}\label{C.mutual}
As in Theorem~\ref{T.1_notin},
let $R$ be a $\!k\!$-algebra, $p$ an element of $R$
such that $1\notin pR+Rp,$ and $B\cup\{1\}$
a basis of $R$ as in~\eqref{d.B=} and~\eqref{d.B_}; and
let us also assume $p\neq 0.$
Let
\begin{equation}\begin{minipage}[c]{35pc}\label{d.R''}
$R''\ =\ R\,\lang q\mid pqp=p,\,qpq=q\rang,$
\end{minipage}\end{equation}
i.e., the $\!k\!$-algebra gotten by adjoining to $R$ a universal
mutual inner inverse $q$ to $p.$

Then $R''$ has a $\!k\!$-basis given by all
words in the generating set $B\cup\{q\}$ which contain no subwords
as in~\eqref{d.xy} or~\eqref{d.xpqpy}
\textup{(}that is, no subwords of the form $xy$ with $x,y\in B$
or $(xp)\,q\,(py)$ with $xp\in B_{++}\cup B_{-+},$
$py\in B_{++}\cup B_{+-}),$ nor any subwords
\begin{equation}\begin{minipage}[c]{35pc}\label{d.qpq}
$q\,p\,q.$
\end{minipage}\end{equation}

The reduction to the above normal form may be accomplished by
the reductions~\eqref{d.xy|->} and~\eqref{d.xpqpy|->} of
Theorem~\ref{T.1_notin}, together with the reduction
\begin{equation}\begin{minipage}[c]{35pc}\label{d.qpq|->}
$q\,p\,q\ \mapsto\ q.$
\end{minipage}\end{equation}
\end{corollary}

\begin{proof}
The normal form described is essentially that of
the case of Theorem~\ref{T.mutual} where $S$ is the
polynomial ring $k[q],$ and $C=C_{++}=\{q^n\mid n>0\}.$
There is the formal difference that words in the basis described
in this corollary may contain strings of the
generator $q,$ while each such string is represented in the
basis gotten from Theorem~\ref{T.mutual} as a single generator $(q^n);$
however, the systems of elements of $R''$ described by the resulting
words are clearly the same.
Likewise, in the indicated case of Theorem~\ref{T.mutual},
the reduction~\eqref{d.qpq|->} is supplemented
by the reductions $(q^m)\,p\,(q^n)\mapsto (q^{m+n-1})$ for all $m,n>0;$
but the reduction~\eqref{d.qpq|->} applied to the
subword $qpq$ of the length-$\!m{+}n{+}1\!$
string $q^m p\,q^n$ clearly has the corresponding effect.

We remark that it would  have been no harder -- but also not
significantly easier -- to verify directly that
adding~\eqref{d.qpq|->} to
the reductions of Theorem~\ref{T.1_notin} yields a reduction
system for $R''$ with all ambiguities resolvable.
\end{proof}

It is easy to supplement Theorem~\ref{T.mutual}
with a normal form result paralleling Proposition~\ref{P.M_norm}
for the $\!T\!$-module induced by
a $\!p\!$-tempered right $\!R\!$-module,
or by a $\!q\!$-tempered right $\!S\!$-module, defined analogously.

\section{Further questions and observations}\label{S.further}

Do results paralleling Theorem~\ref{T.mutual} and
Corollary~\ref{C.mutual} hold without the hypotheses
$1\notin pR+Rp$ and $1\notin qS+Sq$?

For the analog of Corollary~\ref{C.mutual}, where we are only
free to modify $R,$ we can
say ``yes'' in the situation of \S\ref{S.1-sided},
and ``probably'' in that of \S\ref{S.1=:norm}.
In former situation, taking $p$ right invertible
in $R,$ we saw in \S\ref{S.1-sided} that in $R'=R\lang q\mid pqp=p\rang$
our adjoined element $q$ also became a right inverse to $p.$
But this makes $p$ and $q$ mutually inner inverse; so
$R''=R';$ so the additional relation $qpq=q$ and
the reduction $qpq\mapsto q$ have no additional effect, nor does
exclusion of the string $qpq$ from words in our basis.
Thus, for this case the analog of Corollary~\ref{C.mutual} is
trivially true.

For $R$ and $p$
as in \S\ref{S.1=:norm}, hand calculations I have made suggest
that the analog of Corollary~\ref{C.mutual} also holds:
All the ambiguities arising from overlaps between~\eqref{d.qpq|->}
and the reductions \eqref{d.1=:xpqpy|->}, (\xppy{2}),
(\xppy{3}), \eqref{d.1=:xpqtp|->} and~(\xptp{2}) appear
to be resolvable, so it is likely that computations like those
of \S\ref{S.1=:norm} can prove the same for
ambiguities involving~\eqref{d.qpq|->} and any of
the reductions~(\xppy{n}) and~(\xptp{n}).

On the other hand, for Theorem~\ref{T.mutual},
the obvious generalization with $R$ no longer
assumed to satisfy $1\notin pR+Rp,$
while $S$ is still assumed to satisfy $1\notin qS+Sq,$
but not restricted to be $k[q],$ definitely does not hold.
For an extreme example, if $p\in R$ is a nonzero element generating
within $R$ a finite-dimensional field extension $F$ of $k,$
then $F$ will also be generated by $p^{-1},$
hence if an element $q\in S$ is to become an inner inverse
of $p$ in an algebra containing (embedded copies of) both
$R$ and $S,$ the subalgebra of $S$ generated by $q$ must
have the same structure $F;$ which cannot be true if
$1\notin qS\cap Sq$ (and
is very restrictive even if this is not assumed).
To see that there are also obstructions to
the analog of Theorem~\ref{T.mutual} when
$1\in pR+Rp-(pR\cup Rp),$
take $S=k[q\mid q^2=0]$ (which clearly satisfies $1\notin qS+Sq).$
We saw in \S\ref{S.1=} that the relations $pqp=p$ and $1\in pR+Rp$
together imply $pqqp=pq+qp-1$~\eqref{d.1=:pqqp}.
Combining this with the relation $q^2=0$ holding in $S,$ we
get $pq+qp=1.$
But $pq,$ $qp$ and $1$ are distinct words not containing
any subwords \eqref{d.1=:xy}-\eqref{d.1=:xpqtp},
\eqref{d.C:xy} or~\eqref{d.C:xqpqy};
so if the analog of Theorem~\ref{T.mutual} held, they
would be $\!k\!$-linearly independent.
Nor does it help to assume, instead,
that $S$ and $q$ satisfy $1\in qS+Sq-(qS\cup Sq);$
for if we take for $S$
the $2\times 2$ matrix ring over $k,$ and for $q$ the square-zero
matrix $e_{12},$ we get the same problem just described.

But perhaps others will be able to find useful normal form results
for some cases of this construction.
\vspace{.3em}

We end this note by recording an alternative way to construct
the algebra $R''=R\,\lang q\mid pqp=p,\,qpq=q\rang$ from
$R'=R\,\lang q\mid pqp=p\rang,$ implicitly noted in
the original version of~\cite{A+KOM}.
This does not require that our algebras be over a
field, so we assume an arbitrary commutative base ring $K.$

\begin{lemma}[after P.\,Ara and K.\,O'Meara, original version of \cite{A+KOM}]\label{L.retract}
Let $R$ be an algebra over a commutative ring $K,$ let $p$ be
any element of $R,$ and let $R'$ be the
$\!K\!$-algebra $R\,\lang q\mid pqp=p\rang.$

Then $R'$ admits a retraction \textup{(}idempotent $\!K\!$-algebra
endomorphism\textup{)} $\varphi$
that fixes the image of $R,$ and takes $q$ to $qpq.$
The retract $\varphi(R')$ is naturally
isomorphic to $R''=R\,\lang q\mid pqp=p,\,qpq=q\rang,$
via an isomorphism $\psi$ that carries $q\in R''$ to
$\varphi(q)=qpq\in\varphi(R').$
\end{lemma}

\begin{proof}
The defining relation $pqp=p$ of $R'$ clearly implies
the two relations
\begin{equation}\begin{minipage}[c]{35pc}\label{d.p(qpq)p&}
$p\cdot qpq\cdot p=p$\quad and\quad $qpq\cdot p\cdot qpq=qpq.$
\end{minipage}\end{equation}
The first shows that $qpq$ satisfies the relation
over $R$ that is imposed on $q$ in $R';$ hence $R'$ admits an
endomorphism $\varphi$ over $R$ taking $q$ to $qpq,$
and by the second relation, $\varphi$ is idempotent.
Moreover, the relations of~\eqref{d.p(qpq)p&} together show that
the image of $q$ in $\varphi(R')$ satisfies
the relations imposed on $q$ in the definition of $R'';$ so we get a
homomorphism $\psi:R''\to\varphi(R')$ taking $q$ to $\varphi(q)=qpq.$
On the other hand, the factor-map $\theta:R'\to R''$
takes $qpq\in R'$ to $qpq=q\in R'',$ from which it
is easily seen that the restriction of $\theta$ to $\varphi(R')$ is a
$\!2\!$-sided inverse to $\psi,$ establishing the asserted isomorphism.
\end{proof}

So if we know the structure of $R',$ the above lemma gives
us a way of studying $R''.$
However, I have not found it easy to apply this to the description
of $R'$ that we obtained in~\S\ref{S.1=:norm}
for the case $1\in pR+Rp - (pR\cup Rp),$ because substituting
$qpq$ for $q$ in normal-form expressions for elements
of $R'$ gives expressions that are in general not in normal form.
E.g., for $n>1$ the image $\varphi(q^n)$ can be
reduced repeatedly using~\eqref{d.1=:pqqp},
and it is hard to see just what relations such reductions lead to.

\end{document}